\documentclass[11pt]{amsart}
\pdfoutput=1
\usepackage[top=2.5cm, bottom=3cm, left=3cm, right=2.5cm]{geometry}
\usepackage{amsmath,amsthm,amsfonts,amscd,amssymb,amsbsy,epsf}
\usepackage{verbatim,color,enumerate,graphicx}
\newtheorem{theorem}{Theorem}[section]
\newtheorem{remark}{Remark}[section]
\newtheorem{corollary}[theorem]{Corollary}

\newtheorem{proposition}[theorem]{Proposition}
\newtheorem{lemma}[theorem]{Lemma}
\newtheorem{definition}[theorem]{Definition}
\usepackage{macros,mydef}       
\begin{document}
\title [Methane Hydrates]{Advection of Methane in the Hydrate
  Zone: Model, Analysis and Examples}
\author {Malgorzata Peszynska}
\address {Department of Mathematics, Oregon State University, Corvallis, OR 97331}
\email {mpesz@math.oregonstate.edu} 
\author {Ralph E. Showalter}
\address {Department of Mathematics, Oregon State University, Corvallis, OR 97331}
\email {show@math.oregonstate.edu} 
\author {Justin T. Webster}
\address {Department of Mathematics, North Carolina State University, Raleigh, NC 27603
and College of Charlston, Charlston, SC 29424} 
\email {jtwebste@ncsu.edu} 
\date{\today}

\begin {abstract} 
A two-phase two-component model is formulated for the
advective-diffusive transport of methane in liquid phase through
sediment with the accompanying formation and dissolution of methane
hydrate.  This free-boundary problem has a unique generalized solution
in $L^1$; the proof combines analysis of the stationary semilinear
elliptic Dirichlet problem with the nonlinear semigroup theory in
Banach space for an m-accretive multi-valued operator. Additional
estimates of maximum principle type are obtained, and these permit
appropriate maximal extensions of the phase-change relations.  An
example with pure advection indicates the limitations of these
estimates and of the model developed here. We also consider and
analyze the coupled pressure equation that determines the advective
flux in the transport model.
\end{abstract}
\maketitle

\section{Introduction}   \label{intro}

Methane hydrates are crystalline solid compounds consisting of methane
molecules encased in a cage of water molecules. These solids are
stable only at the combined low temperatures and high pressures found
in offshore continental slopes or permafrost regions. Methane hydrates
have been a subject of intense geophysics research for decades due to
their potential as energy sources or as hazards to climate or seafloor
stability \cite{NRuppel03,Sloan,TT04,KoreaHydrates}. The modeling of
methane hydrate formation and stability requires the use of
multi-phase flow models to quantify the exchange of components between
phases in combination with a thermodynamically consistent description
of the dynamic partitioning of these components
\cite{XuRuppel99,LF08,PTT10,PIMA11}. Moreover, the occurence of
hydrates is tied to the availability and type of advective pathways
and the associated permeability and porosity of the host medium
\cite{DD11,JungSantamarina,YunSantamarina,Tohidi2001,ChinaHydrates}. The
mathematical difficulties presented by any realistic model include
systems of partial differential equations with degeneracy and
multi-valued representations of phase change.

In \cite{GMPS14} we took a first step towards the analysis of
well-posedness of a simplified methane hydrate system. We assumed
geothermal and hydrostatic equilibrium, and thus no energy or pressure
equations were necessary. We considered diffusion as the only
transport mechanism, and represented the phase change using
variational inequalities or nonlinear complementarity constraints
which vary with depth. The theory developed in \cite{GMPS14} gives a
time-differentiable solution for which the evolution equation holds in
a space of distributions, $H^{-1}$, but the methods apply only to
self-adjoint (diffusive) form of transport. Furthermore, we defined a
fully implicit in time finite element scheme for the problem and
demonstrated that it converges at the same rate as a similar scheme
for Stefan free-boundary problem.

The system formulated and analyzed in this paper accounts for the
transport of methane by means of both fluid advection and diffusion,
and for the coupled pressure equation which gives advective flux. It
is motivated by observations of massive hydrate deposits which could
not have occured by diffusion only. Rather, a combination of advective
flux together with local biogenic production of methane is required
for the accumulation of such massive deposits over realistic time
scales \cite{TorresICGH,DD11}.  The necessity to include advection
motivated us to go beyond the earlier results of \cite{GMPS14} and
obtain additional estimates on the components of the solution.  In
this paper we shall obtain a solution shich is continuous with values
in the function space $L^1$; although smoother in the spatial
variable, it is formally less smooth in time.
Furthermore, in this paper we account for the pressure equation which is 
coupled to the methane transport model and can be formulated in
several variants. To model the coupled transport-pressure system we
use a staggered-in-time strategy.

This paper is organized as follows.  In Section~\ref{sec:pre} we
introduce some mathematical concepts and notation that will be used
thereafter. In Section~\ref{model} we describe the transport model
which, as in \cite{GMPS14}, is of compositional flow type with two
phases, solid and liquid, and two components, water and
methane. Additionally we describe the pressure equation coupled to the
transport model, which is not covered in \cite{GMPS14}. After some
simplifications, the transport model is a partial differential
equation whose solution is subject to constraints which vary with the
depth. These constraints appear as complementarity conditions or as
variational inequalities on the solution. In addition, we describe
variants of the pressure equation coupled to the model.

The simplified transport model is proven in Section~\ref{analysis} to
be well-posed and to satisfy a useful estimate of maximum principle
type. These results apply to a general class of semilinear
elliptic-parabolic partial differential equations in which the
nonlinearity may depend on the spatial variable.
Section~\ref{example} contains an explicit 1D example (without
diffusion) which indicates some limitations of the model by means of
the blow-up that results from non-homogeneous boundary flux conditions
or large initial data. In Section~\ref{sys} we discuss the
time-dependent system with the transport model coupled to the pressure
equation, whose solution and analysis relies on a loosely coupled
staggered in time scheme.  Section~\ref{sec:conclusions} contains
possible extensions of this work and some work underway.

\section{Notation and Preliminaries}
\label{sec:pre}

Here we introduce some notation and recall the theory that will be
used in the following.
First, we denote the extended real number system by $\Re_\infty
\equiv (-\infty,+ \infty]$.  An extended real-valued function
  $\varphi: \Re \to \Re_\infty$ is {\em convex} if
\ba
\label{eq:varphi}
\varphi(tu + (1-t) v) \le t \varphi(u) + (1-t) \varphi(v)
\text{ for } u,v \in \Re,\ 0 \le t \le 1.
\ea
It is {\em proper} if $\varphi(\xi) < \infty$ for some $\xi \in \Re$
and its {\em effective domain} is the set $\Dom (\varphi) = \{\xi \in
\Re: \varphi(\xi) < \infty \}$.
For such a function, the {\em subgradient} of $\varphi$ at $u \in \Dom(\varphi)$
is the set of all $u^* \in \Re$ such that 
$$u^*(v-u) \le \varphi(v) - \varphi(u) \text{ for all } v \in \Re,$$
and this set is denoted by $\partial \varphi(u)$.  The {\em maximal
monotone} graphs in $\Re \times \Re$ are characterized as the
subgradients of proper convex lower-semicontinuous functions on $\Re$.
These multi-valued relations extend the notion of a continuous
monotone function.  Related results hold in any Hilbert space, but we
shall not need that generality here
\cite{Brezis73,EkelandTemam99,Rockafellar68,Rockafellar69,Showalter97}.

As an example, we consider the set of pairs $x,\,y$ that are related by 
$$ y \le 0,\ x \ge 0,\ y\,x = 0.$$
This arises in the {\em complementarity problem}, a special {\em variational
inequality}, where $y = f - A(x)$ when $f$ and the function $A(\cdot)$ are 
given \cite{KindStam00,ItoKun08,Ulbrich11}.
If we let $I^+(\cdot)$ be the {\em indicator function} of the positive
real numbers, that is, $I^+(x) = 0$ if $x \ge 0$ and $I^+(x) = + \infty$ if
$x < 0$, then $I^+(\cdot)$ is a proper, convex and
lower-semicontinuous function, and the complementarity conditions are
equivalent to
$$ y (z - x) \le I^+(z) - I^+(x) \text{ for
all } z \in \Re\,.$$
This is the {\em subgradient} constraint
$ y \in \partial I^+(x)$
which characterizes the maximal monotone relation $\partial
I^+(\cdot)$ on $\Re \times \Re$.
It is approximated by the derivative $y =\tfrac{d}{dx} I^+_\lambda(x)$
of the regularized indicator function,
\begin{equation*}
I^+_\lambda(x) = \begin{cases} 0 \text{ if } x \ge 0, \\
                   \tfrac{x^2}{2 \lambda} \text{ if } x < 0,
\end{cases}
\text{ for which } \ 
\tfrac{d}{dx} I^+_\lambda(x) = \begin{cases} 0 \text{ if } x \ge 0, \\
                   \tfrac{x}{\lambda} \text{ if } x < 0.
\end{cases}
\end{equation*}
This special case will be used below.

We shall use below the {\em positive part} function, $x^+ =
\tfrac{1}{2}(x + |x|)$, the signum graph, $sgn(x) =
\{\tfrac{x}{|x|}\}$ for $x \neq 0$ and $sgn(0) = [-1,1]$, and the
subgradient of $x^+$, namely, $sgn^+ = \tfrac{1}{2}(1 + sgn(x))$.  We
denote by $sgn_0$ the corresponding (single-valued) function with
$sgn_0(0) = 0$ and similarly with $sgn^+_0$.  Finally, we denote the
gradient of a function $p(\cdot)$ by the (column) vector of partial
derivatives, $\grad p = (\partial_1 p, \dots \partial_N p)^T$ and the
divergence of the vector function $\q(\cdot) = (q_1, \dots q_N)^T$ by
$\div \q = \sum_{j=1}^N \partial_j q_j$.

\subsection*{Measurable-convex integrands}

Let $G$ be an open bounded domain in $\Re^N$. Assume that the
extended-real-valued function $\varphi(x,\xi)$ is a {\em
  measurable-convex integrand}:
\begin{itemize}
\item
for each $x \in G$, the function $\varphi(x,\cdot): \Re \to
\Re_{\infty}$ is proper, lower-semicontinuous and convex, and
\item
for each $\xi \in \Re$, the function $x \mapsto \varphi(x,\xi)$ is
measurable.
\end{itemize}
This notion was developed in \cite{Rockafellar68,Rockafellar69} and
applied in \cite{GMPS14}.

A useful regularization of such functions is the {\em Moreau-Yosida
approximation}: for $\lambda > 0$, set
\begin{equation}   \label{moreau}
\varphi_\lambda (x,r) = 
\inf_{t \in \Re} \{\tfrac{1}{2 \lambda} |r - t|^2 + \varphi(x,t)\}\,.
\end{equation}
Each of these has a derivative, 
$$\beta_\lambda(x,r) = \tfrac{\partial}{\partial r} \varphi_\lambda (x,r),
\ r \in \Re,$$
which is Lipschitz continuous on $\Re$ with constant $1/\lambda$,
and we have monotone convergence
\begin{equation*}
\lim_{\lambda \to 0^+} \varphi_\lambda (x,r) = \varphi(x,r), \  x \in G.
\end{equation*}
\noindent
For each $x \in G$, we denote the subgradient of
$\varphi(x,\cdot)$ by $\beta(x,\cdot) = \partial
\varphi(x,\cdot)$. Such a family of maximal monotone graphs will be
used to formulate our problem.

\subsection*{Accretive operators and initial-value problems}
\begin{definition}
An operator (relation) $\aA$ on a Banach space $X$ is {\em accretive}
if for $[x_j,y_j] \in \aA$, $j=1,2$ and $\lambda>0$, we have
	$$||x_1-x_2|| \le ||(x_1+\lambda y_1)-(x_2+\lambda y_2)||.$$
This is equivalent to requiring that $(I+\lambda \aA)^{-1}$ is a
contraction on $\Rg(I+\lambda \aA)$ for each $\lambda>0$.
An accretive operator $\aA$ is {\em m-accretive} on $X$ if
additionally the range condition $\Rg(I+\lambda \aA) = X$ holds for
every $\lambda > 0$.
\end{definition}

Consider now an m-accretive operator $\aA$ and an evolution equation
\begin{equation}  \label{ageneq}
u'(t) + \aA u(t) \ni F(t),\ \ 0 < t < T,\ u(0) = u_0 \,.
\end{equation}
The nonlinear semigroup generation theorem asserts that if $\aA$ is an
m-accretive operator on the Banach space $X$, the Cauchy problem
\eqref{ageneq} is well-posed \cite{CranEvans75,Evans77,Showalter97}.
It gives a solution which is minimally smooth in time.
\begin{definition}
An {\em $\varepsilon$-solution} of \eqref{ageneq} is a discretization 
\begin{equation}
\mathcal D \equiv \{ 0=t_0<t_1<...<t_N=T;~F_1,...,F_N \in X\}
\end{equation}
and a step function 
\begin{equation}  \label{step}
s(t) \equiv \begin{cases}s_0 & t=t_0 \\ s_j & t \in (t_{j-1},t_j]
\end{cases}
\end{equation} 
for which
\begin{gather*}
t_j-t_{j-1} \le \varepsilon~~\text{for}~~1 \le j \le N,
\\
\sum_{j=1}^N\int_{t_{j-1}}^{t_j}||F(t)-F_j|| dt <
	\varepsilon,~\text{and}
\\
\dfrac{s_j-s_{j-1}}{t_j-t_{j-1}}+\aA(s_j) \ni F_j,~~1 \le j \le N.
\end{gather*}
\end{definition}
The step function \eqref{step} provides a natural approximate solution to
\eqref{ageneq} by backward differences in time.
\begin{definition}
\label{def:C0}
A {\em $C^0$-solution} to \eqref{ageneq} is a function $u \in
C([a,b];X)$ such that for each $\varepsilon>0$ there is an
$\varepsilon$-solution $\mathcal D, s$ of \eqref{ageneq} with
	$$||u(t)-s(t)|| \le \varepsilon.$$
\end{definition}

The nonlinear semigroup theory \cite[p.228]{Showalter97} shows that
the Cauchy problem \eqref{ageneq} is well-posed with this notion
of solution.
\begin{theorem}\label{CL}
Let $\aA$ be $m$-accretive on a Banach space $X$. For each $u_0 \in
\overline{\Dom(\aA)}$ and $F \in L^1(0,T;X)$ there is a unique
$C^0$-solution of the Cauchy problem \eqref{ageneq}.
\end{theorem}
See \cite{Benilan72b,BenCranSacks88,CranLigg71,CranEvans75,Evans77,
  PeszShow98,ShowLitHorn96,LittShow95,show84b} for development and
applications of this theory to problems of structure similar to that
considered in this paper.

The objective in Section~\ref{analysis} is to transform the hydrate
transport model developed in Section~\ref{model} into a form to which
Theorem~\ref{CL} can be applied.

\section{The Model}  \label{model}
%
In this Section we describe the model for methane transport
as well as the coupled pressure equation.

The subseafloor region $G \subset \Re^3$ is a porous sediment of
porosity $\phi$ and permeability $\kappa$ through which the liquid
phase (brine) can flow. This liquid phase may have some methane gas
dissolved in it; the methane component is supplied by microbial
activity, or is supplied from much deeper earth layers.  If the amount
of methane attains a certain maximum amount for a given pressure and
temperature, methane comes out of the liquid solution in form of
either free gas or methane hydrate, and that form depends on the
pressure and temperature conditions. Methane hydrate, an ice-like
substance, forms in conditions of high {pressure} $p(x,t)$ and low
temperature $T(x)$, while free gas forms at higher temperatures or
lower pressures, or if there is not enough water available. It is the
formation of the hydrate and its possible dissociation in the hydrate
zone with abundance of water component that we wish to describe in
this paper. Inclusion of a free gas phase in the model is the subject
of ongoing work and will not be discussed here.

The phases within the pore system are {\em liquid} and {\em hydrate} 
indexed by subscripts $i = \l,\ h$.  {\em Phase saturation} is the
volume fraction $S_i(x,t)$ of phase $i$ present in the pores. Assume
there is no (free) gas phase present here, so these two phases fill
the pore space:
$ S_\l + S_h = 1\,.$
The components are {\em water} and {\em methane} indexed by
superscripts $j = W,\ M$.
The {\em density} of phase $i$ is $\rho_i = \rho_i^W + \rho_i^M,\ i =
\l,\,h$, where $\rho_i^j$ is the {\em mass concentration} of component
$j$ in phase $i$.
The corresponding 
{\em mass fractions} are $\chi_i^M = \frac{1}{\rho_i} \rho_i^M,\ 
                    \chi_i^W = \frac{1}{\rho_i} \rho_i^W,\ i = \l,\,h$,
so we have $0 \le \chi_i^j \le 1$ and $\chi_i^M + \chi_i^W = 1$.
Also we assume abundant water component $\chi_\l^W > 0$. 

\subsection{Transport model with phase constraints}
The {\em mass conservation} equation for the methane component takes
the form
\ba
\label{eq:mass}
\tfrac{\partial}{\partial t} (\phi S_\l \rho_\l \chi_\l^M 
                                  + \phi S_h \rho_h \chi_h^M)
+ \div \J_\l^M = f_M
\ea
in which the flux of the methane in the liquid 
has an {advective} and a diffusive part 
\ba \J_\l^M = \rho_\l \chi_\l^M \q  - \rho_\l D_\l^M \grad \chi_\l^M\,.
\ea
The flux $\q$ is the {\em Darcy velocity}. The molecular diffusion
term $- \rho_\l D_l^M \grad \chi_l^M$ arises from Fick's law, and the
diffusivity $D_{l}^M$ can be scaled as in \cite[2.2-20]{lake} with
porosity and liquid saturation, but will be simplified here by
assuming $D_l^M \equiv const$ as is done in \cite{DD11,GMPS14}.

\subsubsection*{State Equations}   
Additional conditions that are special to the situation studied here
include the following.  The liquid is {\em incompressible}: $\rho_\l =
constant$. 
We also assume that water phase is present everywhere, so $S_\l > 0$
and liquid pressure $p(x,t)$ is defined everywhere. {\em Salt} content of
the brine (liquid phase) $\chi_l^S$ and {\em temperature} $T(x)$ are
assumed to be known and constant in time. The temperature $T(x)$ is
assumed to be linearly increasing with depth, and $\chi_l^S \equiv
const$ is assumed to be that of seawater.
The content of the hydrate phase is fixed, so its properties
$ \rho_h^W,\ \rho_h,\ \rho_h^M,\ \chi_h^W ,\ \chi_h^M $
are all known constants.
Finally, we mention the need as in \cite{DD11} to distinguish between
different rock types of a given sediment by assigning to it a
categorical variable $r(x)$. 

The remaining $\chi_\l^M$ and $S_h = 1 - S_\l$ are essential unknowns.

\subsubsection*{Phase Equilibria} 
Let 
\bas
\chi^*(p(x,t),T(x),\chi_l^S(x,t),r(x))
\eas
denote the {\em maximal mass fraction} of methane that can be
dissolved in the liquid for the given pressure $p$, temperature $T$,
and salinity $\chi_l^S$ in sediment of rock type $r(x)$.  Typically
$\chi^*$ increases with temperature (thus with depth), has only mild
dependence on $p(x,t)$ and $\chi_l^S(x)$, but can vary substantially
between different rock types \cite{DD11}.  (Dependence
of $\chi^*$ on $p(x,t)$ is strong in the gas zone which is not
considered here). Assuming these are known, we see that the methane
{\em maximum solubility constraint} $\chi^*$ can be approximated 
as a function of $x$
\ba
\label{eq:chia}
\chi^*(p(x,t),T(x),\chi_l^S(x,t),r(x)) \approx \chi^*(x).
\ea

Now the hydrate is present only where the liquid is fully saturated, so
$\chi_\l^M = \chi^*(x)$ in the hydrate region.  That is, the
dissolved mass fraction takes the {\em maximal} value wherever $S_h > 0$:
$ S_h > 0$ implies $\chi_\l^M = \chi^*(x). $
Conversely, if the amount of methane does not attain $\chi^*(x)$, then
no hydrate can be present:
$\chi_\l^M < \chi^*(x)$ implies $S_h = 0.$
In summary, the hydrate saturation and liquid mass fraction of methane
satisfy the {\em complementarity constraints} \cite{BenGharbia}
\begin{equation}  \label{phase0} 
\begin{cases}
S_h \ge 0, 
\\
\chi^*(x)  - \chi_\l^M \ge 0,
\\
S_h\,\big( \chi^*(x) - \chi_\l^M \big) = 0.
\end{cases}
\end{equation}
To make the model physically meaningful, we need to have 
\ba
\label{eq:sath}
S_h \leq 1,\;\; \chi_l^M \ge 0.
\ea
Ensuring \eqref{eq:sath} is the crux of the analysis presented in
Section~\ref{analysis} and, as we show, is not always possible. Since
solutions violating \eqref{eq:sath} are unphysical, the question
arises of whether the model is therefore adequate, or whether the
analysis is lacking. These issues are addressed in
Section~\ref{sec:conclusions}.

\subsubsection*{The Transport Equation}
Now we introduce the choice of variables $S,\,\chi$, functions of the
point $x \in G$ and time $t > 0$:
\begin{equation*}
S \equiv S_h(x,t) = 1 - S_\l(x,t), \ \chi \equiv \chi_\l^M(x,t).
\end{equation*}
After division by $\rho_\l$, the mass conservation equation for
methane \eqref{eq:mass} is
\begin{subequations} \label{rs00}
\begin{equation}  \label{mas-con}
\tfrac{\partial}{\partial t} 
( \phi (1-S)  \chi + \phi S\,R )
+ \div ( \q \chi - D_\l^M \grad \chi ) 
=  \tfrac{1}{\rho_\l} f_M 
\end{equation}
with two unknowns $\chi$ and $S$, and where we have set
\ba
\label{eq:rdef}
R \eqdef \tfrac{\rho_h \chi_h^M}{\rho_\l}.
\ea
We can also define for future convenience the (dimensionless) total 
methane content per mass of liquid phase
\ba
\label{eq:udef}
u \eqdef  \phi (1-S)  \chi + \phi S\,R.
\ea

The two variables $\chi$ and $S$ are connected by the phase
equilibrium condition \eqref{phase0} written as a subgradient,
\begin{equation} \label{eq:constraint}
\chi \in \chi^*(x,p) + \partial I^+(S) \,,
\end{equation}
\end{subequations}
where $I^+(\cdot)$ denotes the indicator function of the positive real
numbers.  
For simplicity, we shall assume
\ba
\label{eq:phi-one}
\phi(x,t) =1, \; x \in G,
\ea
but we confirm in Remark~\ref{rem:phi-one} that this assumption is
unnecessary.

Since $S$ is a monotone relation in $\chi$, and since as is known in
practice \cite{LF08,GMPS14},
\ba
\label{eq:defu}
\chi \le \chi^*(x) < R,
\ea
the system \eqref{rs00} is a semi-linear {\em porous medium equation}
  \cite{vazquez07}
\begin{equation} \label{bs00}
\tfrac{\partial}{\partial t} \beta (x,\chi)
+ \div (\q \chi - D_\l^M \grad \chi ) \ni f,
\ x \in G,\ 0 < t < T,
\end{equation}
with advection and an $x$-dependent family of multi-valued monotone
graphs $\beta(x,\cdot)$. The equation \eqref{bs00} is similar to the
{\em Stefan problem}, but with advection and with $x$-dependence of
the constraints. In the Stefan problem the variable $u$ would play the
role of enthalpy, and $\chi$ would be temperature.
The model \eqref{rs00} occurs as equation (3) in \cite{DD11}, and as
part of the comprehensive models developed in \cite{LF08} where
$p(x,t),T(x,t),\chi_l^S(x,t)$ vary and are unknowns. 

The advection-free case of \eqref{rs00} in $\Re^N$ with $\q = \0$ was
analyzed in \cite{GMPS14} in the Hilbert space $H^{-1}(G)$, but the
analysis there depended on the symmetry of the linear elliptic
operator $-\div D_\l^M \grad$ and does not extend to the case $\q
\neq \0$.  The results of \cite{DibeShow81} formally may apply to give
existence of a solution of \eqref{rs00} in $H^{-1}(G)$ when the
elliptic part of \eqref{rs00} is coercive, and uniqueness if
additionally $\q=\0$. However, since the maximum estimate is not
available for these solutions, they have limited interest here.

The objectives in Section~\ref{analysis} are to analyze the
initial-boundary-value problem for the advection-diffusion system
\eqref{bs00} together with a maximum principle. In particular, by
\eqref{eq:udef}, the constraints \eqref{eq:sath} are equivalent to
\ba
\label{eq:satminmax}
0\leq u(x,t) \leq R,
\ea
and deriving estimates on the solution so that the physically
meaningful bound \eqref{eq:satminmax} holds, is a challenge addressed
in Section~\ref{analysis}.  In order to apply these abstract results
to \eqref{bs00}, we shall need to extend the relations $\beta
(x,\cdot)$ to a family of maximal monotone graphs
$\bar{\beta}(x,\cdot)$, and the estimates obtained below will in some
cases assure that our solution satisfies \eqref{eq:satminmax} and so
is independent of these extensions. Other cases require a more general
modeling framework in which the pressure equation is an important
component.

\subsection{The Pressure Equation}  
The {\em pressure} $p(x,t)$ and {\em Darcy velocity} $\q(x,t)$ of the
filtrating liquid are derived by summing mass conservation equations
for all components as in (\cite{lake}, Chapter 2). Since $\chi_l^S$ is
assumed constant, in our case this would be summing \eqref{eq:mass}
plus an equation for $\chi_l^W$. This leads to a simplified version of
the pressure equation in which we drop diffusion terms,
\ba
\tfrac{\partial}{\partial t} ( \phi (\rho_\l S_l + S_h \rho_h))
+ \div (\rho_\l \q) = 0.  \label{darcy_c}
\ea 
Further simplifying and assuming $\phi \approx const, \rho_l \approx
\rho_h$ as well as incompressibility gives 
\begin{subequations}   \label{darcy-e}
\begin{eqnarray}
\div \q = 0, 
\label{darcy_b} 
\end{eqnarray}
The problem is closed with Darcy's law
\ba
\label{eq:darcyf}
\tfrac{\mu}{\kappa} \q  = - (\grad p-\rho_\l\,\g)\,,
\ea
\end{subequations}
where $\g = -\e_3 g$ is the gravity vector.

Superficially, it appears that the coupling between the transport
equation \eqref{eq:mass} and the Darcy flow \eqref{darcy-e} is one way
only in this model due to the simplified form of the pressure equation
and due to \eqref{eq:chia}. A more comprehensive version of pressure
equation such as \eqref{darcy_c} would yield two-way coupling, and may
involve further nonlinearities if, e.g., the dependence of porosity
$\phi$ on the pressure is known, or is modeled by geomechanics
coupling.

More generally, the porosity or permeability may vary with time due to
the deposition of hydrate, $\phi(x,t) = \phi(p(x,t))$ and $\kappa(x,t) =
\kappa(x,S(x,t))$ in the pressure equation \eqref{eq:darcyf}, and the liquid
pressure $p$ and Darcy velocity $\q$ are likewise time-dependent. This
general case may also be included in the more general evolution
pressure equation \eqref{darcy_c}.

\subsubsection*{Hydrostatic pressure and excess pressure}

In \cite{GMPS14,PTT10} we assumed that pressure is hydrostatic, that
is, that the right side of \eqref{eq:darcyf} vanishes and,
consequently, pressure increases linearly with depth according to
hydrostatic gradient, and $\q=\0$. In order to account for nonzero
flux $\q$, we solve \eqref{darcy-e}, but decompose $p(x,t)$ further into
its hydrostatic part $p^0(x)$ and excess pressure $p^*(x,t)$.

The {\em hydrostatic pressure} $p^0(x)$ is
determined by depth 
\begin{subequations}   \label{darcy0}
\ba
\q^0 &=& \0,\\ 
\grad p^0 &=& \rho_\l\,\g\,.
\ea
\end{subequations}
Then the {\em excess pressure} $p^*(x,t)$ associated with
$\q$ satisfies
\begin{subequations}   \label{darcy*}
\ba
\div \q &=& 0,  
\\
\mu \kappa^{-1} \q &=&- \grad p^*.  
\ea
\end{subequations}
The flux $\q(x,t)$ is determined by either the total pressure from
\eqref{darcy-e} or the excess pressure from \eqref{darcy*}. This
decomposition is useful in numerical approximation.

In particular, for a slightly compressible medium, a pressure-porosity
relation $\phi(x,t) = \Phi^*(x,p^*(x,t))$ is determined by the local mechanics
of the medium, but 
\ba
\label{eq:ksat}
\kappa=\kappa(S) 
\ea
is, in general, not known exactly (see \cite{LF08} for some algebraic
approximate formulas). In fact, \eqref{eq:ksat} may be extended by
pressure-stress dependence as well. 

\section{Analysis of Transport Model}  
\label{analysis}

In this section we shall obtain existence-uniqueness of an
$L^1$-solution and maximum estimates for an initial-value problem for
the semilinear equation \eqref{bs00} with (homogeneous) Dirichlet
boundary conditions. These results are obtained for a problem in which
the graphs $\beta(x,\cdot)$ have been extended to maximal monotone
graphs $\bar{\beta}(x,\cdot)$ which agree with $\beta(x,\cdot)$ on the
set of interest.

The general plan is to apply Theorem~\ref{CL} to an abstract version
of \eqref{bs00}.  To do so, in Section~\ref{sec:elliptic} we make
precise the elliptic operator $A_1$ needed in \eqref{bs00} and its
properties on $L^1(G)$. Next we construct the operator $\aA = A \circ
\bar{\beta}^{-1}(x,\cdot)$ from a general operator $A$ of which the
elliptic operator $A_1$ is an example. Here $\bar{\beta}$ has to be
maximal monotone, and for our application it is $x$-dependent. We
handle the $x$-dependent case by the methods of \cite{BrezisStrauss73}
and in Section~\ref{sec:stationary} we supplement these results to
resolve the stationary problem that results from a backward-difference
approximation of \eqref{bs00}, namely,
\ba
\label{eq:stationary}
\bar{\beta}(x,v(x))+ Av(x) \ni f(x).
\ea
These results show that $\aA$ is m-accretive, and we describe
comparison and maximum estimates for the stationary problem. In
Section~\ref{sec:evolution} we put together the properties of $\aA$
and the abstract nonlinear semigroup theory theory from
Section~\ref{sec:pre} to conclude well-posedness of \eqref{bs00} with
the extension $\bar{\beta}$ of $\beta$. Related estimates are
formulated for the evolution problem.

The use of $\bar{\beta}(x,\cdot)$ in \eqref{bs00} instead of
$\beta(x,\cdot)$, which is not maximal, requires some a-priori
assumptions on the solution. The comparison and maximum principles
show where such a-priori conditions can be eliminated, as they are
consequences of the data.  The final result of this Section is
Proposition~\ref{prop:evolution} which applies the abstract results to
obtain well-posedness of \eqref{bs00}. In Section~\ref{example} we
provide explicit examples where the a-priori conditions can and cannot
be eliminated.

\subsection{Elliptic operator $A_1$ and its properties}
\label{sec:elliptic}

Define the usual continuous bilinear form on
the Sobolev space $V = H^1_0(G)$ corresponding to an
advection-diffusion-reaction problem
\begin{equation} \label{form}
\A v(\psi) = 
\sum_{i,j=1}^N  \int_G a_{ij}(x)\partial_i v \partial_j \psi \,dx
         - \sum_{j=1}^N \int_G q_j(x) v \partial_j \psi \,dx + \int_G a(x) v \psi \,dx,
         \ v,\psi \in V.
\end{equation}
Assume the coefficients
$a_{ij},\ q_j \in C^1(\overline{G}),\ a \in L^\infty(G)$ satisfy
\begin{equation}
\label{eq:assum}
a(x) \ge 0,\ 2 a(x) + \partial_j q_j \ge 0,
\ a_{ij}(x) \xi_i \xi_j \ge c_0 |\xi|^2, \ x \in G,\ \xi \in \Re^N.
\end{equation}
\begin{remark}
\label{rem:bs00}
In \eqref{bs00} we have $a\equiv 0,\nabla \cdot \q=0$, and $a_{ij}=D_l^M\delta_{ij}$ hence \eqref{eq:assum} holds. 
\end{remark}
Now define 
\begin{equation}
\label{eq:A1}
A_1 v = - \sum_{i,j=1}^N \partial_j(a_{ij}\partial_i v) 
+ \sum_{j=1}^N \partial_j(q_j v) + a v.
\end{equation}
with 
$\Dom(A_1) \equiv \{v \in W^{1,1}_0(G): A_1 v \in L^1(G)\}$ where
$A_1 v = f \in L^1(G)$ corresponds to the Dirichlet problem
\begin{equation*}
v \in W^{1,1}_0(G): \A v(\psi) = \int_G f\psi\,dx, \quad \psi \in C_0^\infty(G)\,.
\end{equation*}
Note that \eqref{form} is well defined for $(v,\,\psi) \in  
(W^{1,1}_0(G), C_0^\infty(G))$ and determines $A_1v$. 
Brezis and Strauss \cite{BrezisStrauss73} showed that the operator
$A_1$ has the following properties: 
\begin{proposition}\label{prop:BS} \cite{BrezisStrauss73}
The linear operator $A_1$ is the $L^1(G)$-closure of the
restriction $\A: H^1_0(G) \to L^2(G) \subset H^1_0(G)'$ and it satisfies
\begin{enumerate}[(A)]
\item $\Dom(A_1)$ is dense in $L^1(G)$ and $(I + \lambda A_1)^{-1}$ is
a contraction for each $\lambda > 0$.
\item $\Dom(A_1) \subset W^{1,p}_0(G)$ for any $p$: $1 \le p < N/(N-1)$ and
there is a $c(p) > 0$ such that $c(p) \|v\|_{W_0^{1,p}} \le
\|A_1(v)\|_{L^1}$ for $v \in \Dom(A_1)$.
\item $\sup_G (I + \lambda A_1)^{-1}f \le \max\{0,\sup_G f\}$ for each
$f \in L^1(G)$.
\end{enumerate}
\end{proposition}

These properties of operator $A_1$ are used in \cite{BrezisStrauss73}
to study the stationary problem of structure similar to
\eqref{eq:stationary}. In fact, remarks in \cite{BrezisStrauss73}
cover the $x$-dependent case but require
\ba
\label{eq:meas}
\mathrm{measurability\  of\ the\ resolvents}\ 
(I+\bar{\beta}(x,\cdot))^{-1}
\ea
which is cumbersome to verify for our problem \eqref{bs00}.

In what follows we will use (A), (C), and (B) for $p=1$ of
Proposition~\ref{prop:BS}, and handle the $x$-dependence of $\bar{\beta}$
differently than in \cite{BrezisStrauss73}.

\subsection{The Stationary Problem}
\label{sec:stationary}

We will show now that the proof from \cite{BrezisStrauss73} concerning
the stationary problem \eqref{eq:stationary} for the case of a single
maximal monotone $\bar{\beta}(\xi) = \partial \varphi(\xi)$ and an
abstract operator $A$ extends to the $x$-dependent case without
\eqref{eq:meas} but under some additional assumptions which place the
measurability hypotheses directly on the $\varphi(x,\xi)$ instead of
on the resolvent.  This facilitates checking the hypotheses and allows
the application of the result to \eqref{bs00}. We also prove an
estimate of maximum principle type which is useful later in the
analysis of the evolution problem. The maximum estimate obtained
below bounds not only the values of $u$ but also those of $\chi$. The
results are put together in Theorem~\ref{BS-T1} and its corollaries below.

We start by providing the construction of $\bar{\beta}$ as a
subgradient of $\varphi(x,\cdot)$. For our purposes,
$\varphi(x,\cdot)$ has the domain $\Re$ for each $x \in G$.

\begin{definition}
Assume that
$\varphi(x,\xi)$ is a {\em measurable-convex integrand} with each
$\varphi(x,\xi) \in [0,+\infty)$ and $\varphi(x,0) = 0$.  For each $x
  \in G$, denote the subgradient of $\varphi(x,\cdot)$ by
\ba
\label{eq:beta}
\bar{\beta}(x,\cdot) = \partial \varphi(x,\cdot).
\ea
\end{definition}

\begin{theorem} \label{BS-T1}
Let $\bar{\beta}(x,\cdot)$ be given as in \eqref{eq:beta}. Assume additionally
that
\begin{equation} \label{mc-est}
M_C(x)=\sup\{|u|:~u \in \bar{\beta}(x,v),~|v|\le C\} \in L^2(G)
\text{ for each } C>0.
\end{equation}
Let the linear operator $A: \Dom(A) \to L^1(G)$ satisfy the following:
\begin{enumerate}[(a)]
\item $\Dom(A)$ is dense and $(I+ \lambda A)^{-1}$ is a contraction on
$L^1(G)$ for each $\lambda > 0$;
\item  There is a $c > 0$ such
  that $c \|v\|_{L^1} \le \|Av\|_{L^1}$ for $v \in \Dom(A)$.
\item  $\sup_G (I + \lambda A)^{-1}f \le (\sup_G f)^+$ for each
$f \in L^1(G)$ and $\lambda > 0$;
\end{enumerate}
Then for each $f \in L^1(G)$ there is a unique solution $v \in
\Dom(A),\ u \in L^1(G)$ to the stationary problem
\begin{equation} \label{semi}
u + Av = f \text{ and } u(x) \in \bar{\beta}(x,v(x))\,, 
\text{ a.e. } x \in G\,.
\end{equation}
In addition, if $u_1,v_1$ and $u_2,v_2$ are solutions
corresponding to $f_1, f_2$, then the comparison estimates
\begin{equation} \label{L1-order}
\|(u_1 - u_2)^+\|_{L^1} \le \|(f_1 - f_2)^+\|_{L^1},\ 
\|(u_1 - u_2)^-\|_{L^1} \le \|(f_1 - f_2)^-\|_{L^1},
\end{equation}
hold, and, consequently 
\begin{equation} \label{L1-est}
\|u_1 - u_2\|_{L^1} \le \|f_1 - f_2\|_{L^1},
\end{equation}
i.e., the map $f \mapsto u$ is a contraction on $L^1(G)$.
\end{theorem}

The proof of this Theorem follows a sequence of steps. First, we
recall the following result from \cite{BrezisStrauss73} which provides
key estimates there and below.  Such a result holds only for a single
convex function.
\begin{lemma} \label{key-est} (\cite{BrezisStrauss73}, Lemma 2; 
Prop. II.9.3 in \cite{Showalter97}).  Let the operator $A$ satisfy the
conditions (a), (c) in Theorem~\ref{BS-T1}, and assume the function
$\varphi: \Re \to [0,+\infty]$ is proper, convex and lower
semicontinuous with $\varphi(0) = 0$. Then for each pair $v \in
L^p(G),\ u \in L^{p'}(G),\ Av \in L^p(G),$ and $u(x) \in \partial
\varphi(v(x)) \text{ a.e. } x \in G$, with $p \ge 1$, we have
\begin{equation*}  \label{angle}
\int_G Av(x) u(x)\,dx \ge 0.
\end{equation*}
\end{lemma}

For an $x$-dependent family of such functions, we begin with the
following elementary but useful observation.
\begin{lemma} \label{integrable}
Assume that $\varphi(x,\xi)$ is a measurable-convex integrand with
each $\varphi(x,\xi) \in [0,+\infty]$ and $\varphi(x,0) = 0$. If $w: G
\to [0,+\infty]$ is measurable, then $\varphi(x,w(x))$ is measurable.
If $p \ge 1,\ v \in L^p(G),\ u \in L^{p'}(G),$ and $u(x) \in \partial
\varphi(x,v(x)) \text{ a.e. } x \in G$, then $\varphi(\cdot,v(\cdot))
\in L^1(G)$.
\end{lemma}
\begin{proof}
If $w$ is measurable then from the definition \eqref{moreau} it
follows that each Moreau-Yosida approximation
$x \mapsto \varphi_\lambda(x,w(x))$ is measurable, and these converge
monotonically to $\varphi(x,w(x))$ as $\lambda \to 0$, so $x \mapsto
\varphi(x,w(x))$ is measurable.  With $u,v$ as indicated, we have
$u(x) (0 - v(x)) \le \varphi(x,0) - \varphi(x,v(x))$, and this implies
the integrable upper bound in $0 \le \varphi(x,v(x)) \le u(x) v(x)$.
\end{proof}

\begin{proof}[{\it Proof of Theorem~\ref{BS-T1}}]  We follow the structure of the
proof of Theorem 1 of \cite{BrezisStrauss73}. (The latter is Theorem
II.9.2 of \cite{Showalter97}.). Each step is verified for the new
hypotheses.

Uniqueness of a solution is obtained from the
estimate \eqref{L1-est} and the injectivity of $A$. To verify
\eqref{L1-order}, let $u_1,v_1$ and $u_2,v_2$ be solutions of
\eqref{semi} corresponding to $f_1,\,f_2$. Subtract these two
equations and multiply by $\sigma = sgn_0^+(u_1-u_2 + v_1-v_2)$. Since
$\sigma \in sgn^+(v_1-v_2)$ (and $sgn^+$ does not depend on $x \in
G$), we can apply Lemma~\ref{key-est} to get $\int_G
A(v_1-v_2)\sigma\,dx \ge 0$. Also we have $\sigma \in sgn^+(u_1-u_2)$,
so the first of the estimates \eqref{L1-order} follows.  The second is
obtained similarly by using $sgn^-$. These imply \eqref{L1-est} and as
in \cite{BrezisStrauss73} that the range of $A + \bar{\beta}(\cdot)$ is
closed.

To prove the existence of an approximate solution of \eqref{semi}, let
$\eps > 0$ and $f_\eps \in L^1(G) \cap L^\infty(G)$ be fixed. (The
general case $f \in L^1(G)$ follows later). For each $\lambda > 0$
consider the approximating equation
\begin{equation}   \label{approx}
\eps v_\lambda + Av_\lambda + \bar{\beta}_\lambda(\cdot,v_\lambda) = f_\eps \,,
\end{equation}
where we have regularized $A$ by addition of $\eps I$. 
This is equivalent to
\ba
\label{eq:ss}
v_\lambda = (1 + \lambda \eps)^{-1} (I+\tfrac{\lambda}{1 + \lambda \eps}A)^{-1}
(\lambda f_\eps + (I + \lambda \bar{\beta})^{-1}v_\lambda).
\ea
The right side of \eqref{eq:ss} is a strict contraction in $L^1 \cap
L^\infty$, because it is a composition of two contractions followed by
scaling by a number $(1+\lambda\eps)^{-1} < 1$.

Thus \eqref{eq:ss} has a unique fixed point, $v_\lambda$, a solution
of \eqref{approx} which depends on $\eps > 0,\ \lambda > 0$. Use
Lemma~\ref{key-est} to test \eqref{approx} with $w = sgn_0(v_\lambda)
\in sgn(\bar{\beta}_\lambda(\cdot,v_\lambda))$ to obtain
$$ \eps \|v_\lambda\|_{L^1} + \|\bar{\beta}_\lambda(\cdot,v_\lambda)\|_{L^1} \le \|f_\eps\|_{L^1}.$$
Note that the function $sgn_0(\cdot)$ used to construct the test
function above is independent of $x$, so we can use Lemma~\ref{key-est}.
Moreover, in the norm $\|\cdot\|$ of $L^1 \cap L^\infty$, we have from
\eqref{eq:ss} that
$$ \|v_\lambda\| \le (1 + \lambda \eps)^{-1}(\lambda \|f_\eps\| + \|v_\lambda\|),$$
which implies $\|v_\lambda\| \le \tfrac{1}{\eps}\|f_\eps\|$.

It remains to obtain estimates on $\bar{\beta}_\lambda(\cdot,v_\lambda)$.
From \eqref{mc-est} and the preceding estimate, $|v_\lambda(x)| \le C$
for $C = \tfrac{1}{\eps}\|f_\eps\|$, so we get
\begin{equation}
|\bar{\beta}_\lambda(x,v_\lambda(x)| \le M_C(x),\ x \in G.
\end{equation}
Hence, the sequence $\{\bar{\beta}_\lambda(\cdot,v_\lambda)\}$ is bounded in
$L^2(G)$, and we follow steps identical to those in
\cite{BrezisStrauss73}.  First we obtain limits $v_\lambda \to
v_\eps,\ \bar{\beta}_\lambda(\cdot,v_\lambda) \to u_\eps$ as $\lambda \to
0$. Note here that we have strong limits in $L^2(G)$ due to the result
from \cite{CranPazy1969}. 
These limits satisfy
\begin{equation}   \label{approx0}
\eps v_\eps + Av_\eps + u_\eps = f_\eps,\ u_\eps \in \bar{\beta}(\cdot,v_\eps).
\end{equation}

Finally, for a general $f \in L^1(G)$, we approximate it with a sequence
in $L^1 \cap L^\infty$, $f_\eps \to f$ in $L^1(G)$, solve
\eqref{approx0} for each $\eps > 0$, and then we let $\eps \to 0$ to
get $v_\eps \to v$ and $u_\eps \to u$ in $L^1(G)$ which satisfy
\eqref{semi}.
\end{proof}

Next we prove crucial comparison and maximum estimates. 
\begin{corollary}  \label{BS0-comp}
If $u_1,v_1$ and $u_2,v_2$ are solutions corresponding to $f_1,f_2$
and $f_2 \geq f_1$, then $u_2 \geq u_1$ and $v_2 \geq v_1$. 
\end{corollary}
The first inequality follows from \eqref{L1-order}. The second holds for the
respective approximations by \eqref{approx0}, and hence for their limits.

\begin{proposition} \label{BSmax}
If $f \in L^1(G) \cap L^\infty(G)$ and $k_1 \le 0 \le k_2$, then for
any measurable selections $b_1(x) \in \bar{\beta}(x,k_1),\, b_2(x) \in
\bar{\beta}(x,k_2)$ the solution $v_\eps,\,u_\eps$ of \eqref{approx0}
satisfies the estimates
\begin{subequations}
\begin{gather}
\eps  \|(v_\eps - k_2)^+\|_{L^1} + \|(u_\eps - b_2)^+\|_{L^1} \le
\|(f-b_2)^+\|_{L^1},  \label{pos-est-0}
\\
\eps  \|(k_1 - v_\eps)^+\|_{L^1} + \|(b_1 - u_\eps)^+\|_{L^1}
\le \|(b_1 - f)^+\|_{L^1}.  \label{neg-est-0}
\end{gather}
\end{subequations}
\end{proposition}
\begin{proof}
Let $k_2 \ge 0$ and subtract $\eps k_2$ from the left side and
$b_2$ from both sides of \eqref{approx0} to get
\begin{equation*}
\eps (v_\eps - k_2) + Av_\eps + u_\eps - b_2
\le  f - b_2.
\end{equation*}
(Note that \eqref{mc-est} implies $b_2 \in L^2(G)$.)  Multiply by the
non-negative $w(x) = sgn_0^+(v_\eps(x)-k_2 + u_\eps(x) - b_2(x)) \in
sgn^+(v_\eps(x) - k_2) \cap sgn^+(u_\eps(x))-b_2(x))$ to obtain
\begin{equation*}
\eps (v_\eps(x)-k_2)^+ + Av_\eps(x)w(x) 
+ (u_\eps(x) - b_2(x))^+
\le  (f(x) - b_2(x))^+
\end{equation*}
and use Lemma~\ref{key-est} to integrate and get the
first estimate. The second is proved similarly.
\end{proof}
\begin{corollary} [Maximum estimate]  \label{BS0-max}
In the situation of Theorem~\ref{BS-T1} with $f \in L^1(G)$, assume
$0 \le k$ and that $b(x) \in
\bar{\beta}(x,k)$ is a corresponding measurable selection.

If $f(x) \le b(x)$ a.e. in $G$, then $v(x) \le k$ and $u(x) \le
b(x)$ a.e. in $G$.
\end{corollary}
\begin{proof}
Choose the approximations $f_\eps$ to satisfy the same constraint as
$f$.  Then Proposition~\ref{BSmax} shows the approximating solutions
$v_\eps,\,u_\eps$ of \eqref{approx0} satisfy the desired estimates,
and the same then holds for their $L^1$-limits, $v$ and $u$.
\end{proof}

\begin{remark} Corollary~\ref{BS0-max} 
does not follow from the comparison principle~\ref{BS0-comp}, since
$k,b(x)$ do not need to be solutions of the boundary-value
problem. When $\bar{\beta}$ is independent of $x$, the selection $b(x)$
can be replaced by any constant of appropriate sign to obtain
$L^\infty$-estimates. For example, if $b \in \Rg(\bar{\beta})$ we choose
$k \in \Re$ with $b \in \bar{\beta}(k)$, while for $b > \Rg(\bar{\beta})$
the result is vacuously true.
\end{remark}

\subsection{The Evolution Equation}
\label{sec:evolution}

Now we consider the evolution partial differential equation
\eqref{bs00} with homogeneous Dirichlet boundary conditions. A
solution of \eqref{bs00} written in terms of $u(x,t) \in
\bar{\beta}(x,\chi(x,t))$ satisfies
\ba
\label{eq:bs00}
\tfrac{\partial u}{\partial t} + A \circ \bar{\beta}^{-1}(\cdot,u) \ni F, \ 0 < t < T,
\ea 
with the operator $A$ and monotone graphs $\bar{\beta}(x,\cdot)$ as
defined in Section~\ref{sec:stationary}.  We recall again the
modification $\beta \rightarrow \bar{\beta}$ needed for theory, and
that a general operator $A$ or the particular operator $A_1$ can be
used. In the latter case, \eqref{eq:bs00} corresponds to 
\eqref{bs00} with the maximal monotone extension $\bar{\beta}$ of ${\beta}$.

Now \eqref{eq:bs00} can be written as the abstract Cauchy problem
\eqref{ageneq} provided we identify $\aA$ and demonstrate its
properties required by Theorem~\ref{CL}.

The extended Brezis-Strauss Theorem~\ref{BS-T1} developed in
Section~\ref{sec:stationary} provides the construction of the
appropriate operator 
\ba
\aA = A \circ\bar{\beta}^{-1}(\cdot,\cdot)
\ea
in $L^1(G)$. 
Define the relation $\aA$ on $L^1(G)$ by $\aA(u) \ni f$ if $u \in
L^1(G), f \in L^1(G)$ and that for some $\chi \in \Dom(A)$,
\begin{equation*} 
A \chi = f \text{ and } u(x) \in \bar{\beta}(x,\chi(x))\,, \text{ a.e. } x \in G\,.
\end{equation*}
The Cauchy problem \eqref{ageneq} with this operator $\aA$ is equivalent to the abstract
problem which can be rewritten as
\begin{equation}  \label{geneq}
u'(t) + A \chi(t) = F(t),\ u(t) \in \bar{\beta}(\cdot,\chi(t)),
\ 0 < t < T,\ u(0) = u_0 \,.
\end{equation}

To show $\aA$ is m-accretive, we use results of
Section~\ref{sec:stationary}. The equation \eqref{semi} is equivalent
to $u + \aA(u) \ni f$, and Theorem~\ref{BS-T1} implies that the map $f
\mapsto u$ is a contraction defined on $L^1(G)$. Moreover, the same
holds with $A$ replaced by $\lambda A$ for any $\lambda > 0$, so $\aA$
is m-accretive in the Banach space $L^1(G)$. Thus Theorem~\ref{CL}
applies, and we have the following result.

\begin{proposition} \label{prop:evolution} 
In the situation of Theorem~\ref{BS-T1}, the corresponding
initial-value problem \eqref{geneq} is well-posed. That is, for each
$u_0 \in \overline{\Dom(\aA)}$ and $F(\cdot) \in L^1(0,T;L^1(G))$
there is a unique $C^0$~solution of \eqref{eq:bs00} with $u(0) = u_0$.
\end{proposition}

We continue now to derive estimates on the solution to \eqref{eq:bs00}
which help to determine whether the a-priori extension $\beta
\rightarrow \bar{\beta}$ limits the applicability of
Proposition~\ref{prop:evolution}, namely, whether the solution of
\eqref{eq:bs00} satisfies \eqref{bs00}. This is the case if we can
show that the solution satisfies \eqref{eq:satminmax}.

\begin{corollary} [Comparison principle]
\label{cor:10half}
If $u_1(t),v_1(t)$ and $u_2(t),v_2(t)$ are solutions of the
initial-value problem \eqref{geneq}, with the corresponding data
$u_1(0),F_1(t)$ and $u_2(0),F_2(t)$, then
\bas
\|(u_1(t)-u_2(t))^+\|_{L^1} \leq 
\|(u_1(0)-u_2(0))^+\|_{L^1} \\
\nonumber
+\int_0^t \|(F_1(s)-F_2(s))^+\|_{L^1} ds, \, 0\leq t \leq T,
\eas
and similar inequalities hold for $\|(u_1(t)-u_2(t))^-\|_{L^1}$
and $\|u_1(t)-u_2(t)\|_{L^1}$.
\end{corollary}
\begin{proof}
This follows immediately for the approximations \eqref{step} by the
estimates \eqref{L1-order}.
\end{proof} 

Corollary~\ref{cor:10half} and Corollary~\ref{BS0-max} yield bounds on
a solution as follows.
\begin{remark}
\label{rem:bounds}
Let 
\begin{subequations}
\label{eq:bounds}
\ba v_2 \in \Dom(A)\mathrm{\ with\ }Av_2 =F_2 \geq 0,\mathrm{\ and\ }
\\ u_0(x) \leq u_2(x) \in \bar{\beta}(x,v_2(x))\mathrm{\ in\ }L^1(G).
\ea
\end{subequations}
From Corollary~\ref{cor:10half} we find that if $F \leq
0$, then the solution of \eqref{geneq} satisfies
\ba
u(t) \leq u_2,\;\; \chi(t) \leq v_2,\;\; 0 \leq t \leq T.
\ea
We also find that if $F \equiv 0$, then 
\ba
\label{eq:uplus}
u_0\geq 0 \implies u(t)\geq 0.
\ea
\end{remark}

Similarly, we obtain a {\em maximum estimate} for the initial-value
problem for $\aA = A \circ\bar{\beta}^{-1}(\cdot,\cdot)$.
\begin{corollary} [Maximum estimate]
\label{EVmax}
If $F \le 0$,
and 
\begin{subequations}
\label{eq:binit}
\ba
0 \le k,\, b(x) \in \bar{\beta}(x,k)
\ea 
is a measurable selection, and 
\ba
\label{eq:uinit}
u_0(x) \le b(x) \,a.e., 
\ea 
\end{subequations}
then the $C^0$-solution of the Cauchy
problem \eqref{geneq} satisfies
\ba
u(x,t) \le b(x), \;\; \chi(x,t) \le k,\; a.e. x \in G,
\ea
for $0 \le t \le T$.
\end{corollary}
\begin{proof}
This follows immediately for the approximations \eqref{step} by the
estimates of Corollary~\ref{BS0-max}.
\end{proof} 

It follows from Corollary~\ref{EVmax} that the solution of
\eqref{geneq} is completely independent of those values of $u \in
\bar{\beta}(x,\chi)$ with $u \ge b$ or $\chi \ge k$.  In other words,
we can extend ${\beta}(x,\cdot)$ to a maximal monotone graph
$\bar{\bar{\beta}}(x,\cdot)$ in any (monotone) way for $u \ge b,\,\chi
\ge k$.

\begin{remark}
\label{rem:bottomline}
If the initial data for the problem \eqref{bs00} can be shown to
satisfy \eqref{eq:binit} for some useful pair $k,b(x)$, then the
solution to \eqref{eq:bs00} remains bounded by the same pair. Thus it
does not matter how the graph $\beta$ was extended to $\bar{\beta}$
beyond $k,b$. We can conclude the well-posedness for the problem
\eqref{bs00} with the original $\beta(x,\cdot)$ for the initial data
satisfying \eqref{eq:binit}.
\end{remark}

The remaining difficulty is to identify whether we can find useful
bounds $b,k$ which correspond to physically meaningful solutions, in
particular those that yield solutions that satisfy
\eqref{eq:satminmax}. Examples shown in Section~\ref{example} address
this question.

\subsection{Handling nonconstant $\phi$ and nonhomogeneous boundary conditions}

The analysis given above was formulated for homogeneous boundary
conditions and for constant porosity coefficient set as in
\eqref{eq:phi-one}. 

\begin{remark}
\label{rem:phi-one}
One can treat nonconstant porosity coefficient as follows. If $\phi
\in L^\infty(G)$ and $\phi(x) >0$ for a.e. $x \in G$, then $u \in
\partial \varphi(x,v)$ is equivalent to $\phi(x)u \in \partial
\left(\phi(x) \varphi(x,v)\right)$ for $x \in G$, and the functions
$\phi(x) \varphi(x,v)$ and $\phi(x) \beta(x,v)$ have the same
respective properties as $\varphi(x,v)$ and $\beta(x,v) = \partial
\varphi(x,v)$.

The case of $\phi(x,t)$ is important e.g., since $\phi$ depends on the
pressure, but is considerably more difficult from analysis point of
view and will not be discussed here.
\end{remark}

Next, we discuss boundary conditions. For linear smooth problems the
extension of analysis to non-homogeneous Dirichlet conditions (needed
in applications) is straightforward. For \eqref{bs00} this is also
true, but shifting of boundary conditions affects various maximum and
comparison principles. We address the case of nonhomogeneous boundary
conditions in detail for completeness, and do not revisit analyses
formulated above.

%
Suppose we want to resolve the evolution equation
\begin{equation}  \label{non-hom-ev}
\tfrac{d}{dt} u(t) + A_1 v(t) = F(t),\ u(t) \in \bar{\beta}(\cdot,v(t)),
\end{equation}
in $L^1(G)$ with given initial value $u(0)$, where $A_1$ is the partial
differential operator above but with non-homogeneous boundary
conditions on $v(t)$, independent of $t$.  Let $v_0 \in W^{2,1}(G)$ be a smooth
function that satisfies those boundary conditions, that is, $v$ and
$v_0$ have the same trace on $\partial G$ and $A_1 v_0 \in
L^1(G)$. Choose $u_0 \in L^1(G)$ to satisfy $u_0(x) \in
\bar{\beta}(x,v_0(x))$ for $x \in G$. Then define the translates
\begin{eqnarray*}
\tilde \beta(x,\xi) = \bar{\beta}(x,v_0(x) + \xi) - u_0(x),\ \xi \in \Re,
\\
\tilde u(t) = u(t) - u_0,\ 
\tilde v(t) = v(t) - v_0,\ 
\tilde F(t) = F(t) - A v_0
\end{eqnarray*}
Each $\tilde \beta(x,\cdot)$ is maximal monotone,
$\tilde \beta(x,0) = \bar{\beta}(x,v_0(x)) - u_0(x) \ni 0$,
$\tilde v = 0$ on $\partial G$, and
\begin{equation}  \label{tr-non-hom-ev}
\tfrac{d}{dt} \tilde u(t) + A_1 \tilde v(t) = \tilde F(t),\
\tilde u(t) \in \tilde \beta(\cdot,\tilde v(t)).
\end{equation}
Conversely, if $\tilde u(t),\tilde v(t)$ is a solution of
\eqref{tr-non-hom-ev} with $\tilde v(t) \in \Dom(A_1)$, then $u(t),v(t)$
is a solution of \eqref{non-hom-ev} with the prescribed boundary
values $v(t)|_{\partial G} = v_0|_{\partial G}$.

To apply this observation to get estimates on solutions of
\eqref{non-hom-ev} from those we have on the equation
\eqref{tr-non-hom-ev} with homogeneous boundary conditions, it is
useful to note that we can always choose $v_0$ to have the same upper
or lower bounds as its trace on the boundary. For example, we could
choose $v_0$ to be a harmonic function with the prescribed boundary
values.

\section{Examples of where theory applies and where it does not}  \label{example}

In this Section we illustrate the well-posedness results from
Section~\ref{analysis} and their limitations. In particular, we
exhibit explicit examples and applications of the maximum and
comparison principles. This is the task that was outlined
after Remarks~\ref{rem:bounds} and~\ref{rem:bottomline}. The goal is
to use the comparison and maximum principles to verify
\eqref{eq:satminmax}. This may be possible by
putting appropriate restrictions on the data for some cases, and then
the extension $\beta \rightarrow \bar{\beta}$ does not change the
problem, and we have full well-posedness in those cases, with
solutions satisfying physically meaningful bounds
\eqref{eq:satminmax}.

\begin{figure}
\includegraphics[width=6.5cm]{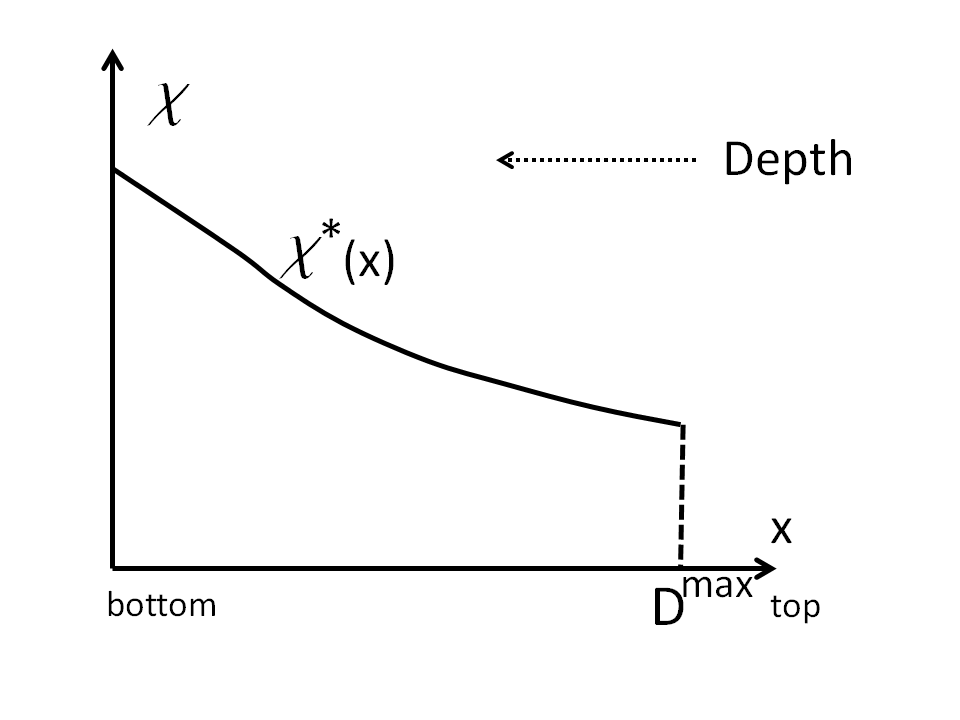}
\includegraphics[width=6.5cm]{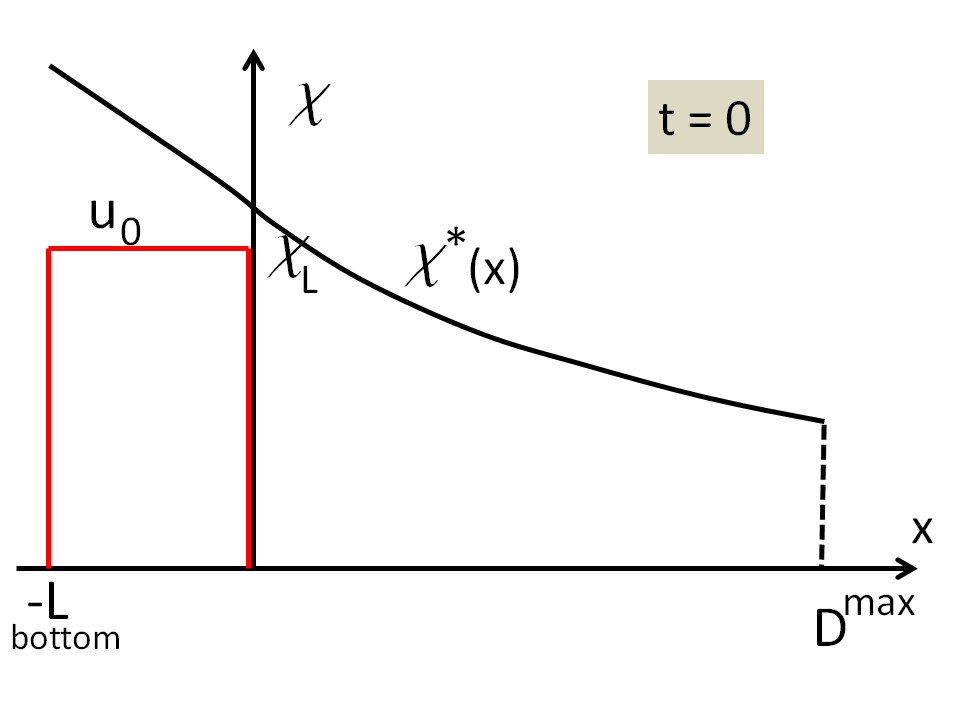}
\centerline{\hfill (a) \hfill (b) \hfill}

\includegraphics[width=6.5cm]{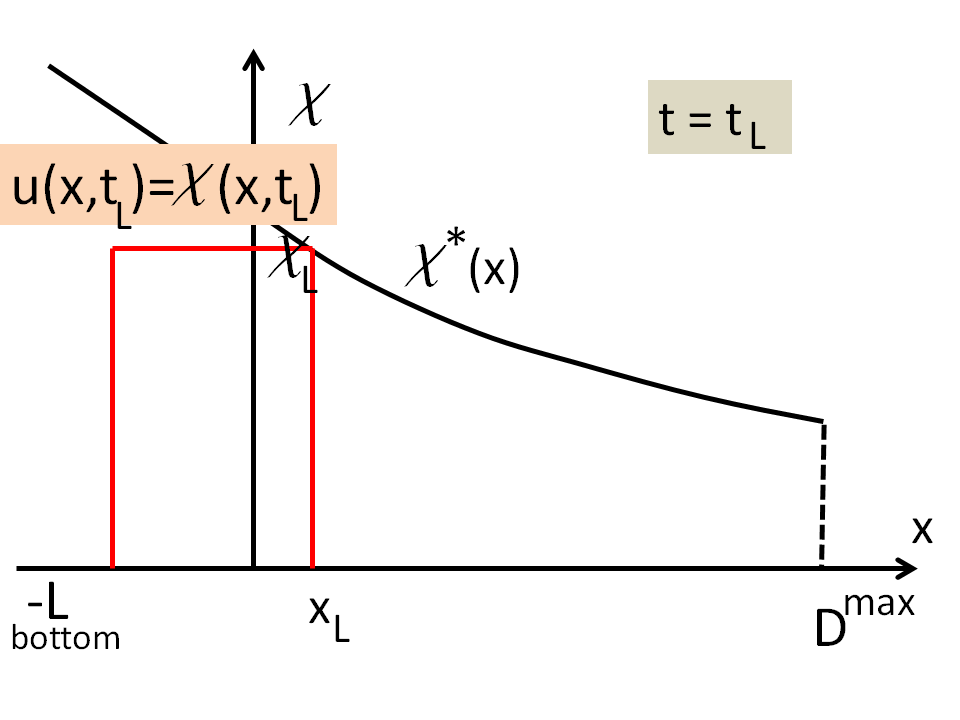}
\includegraphics[width=6.5cm]{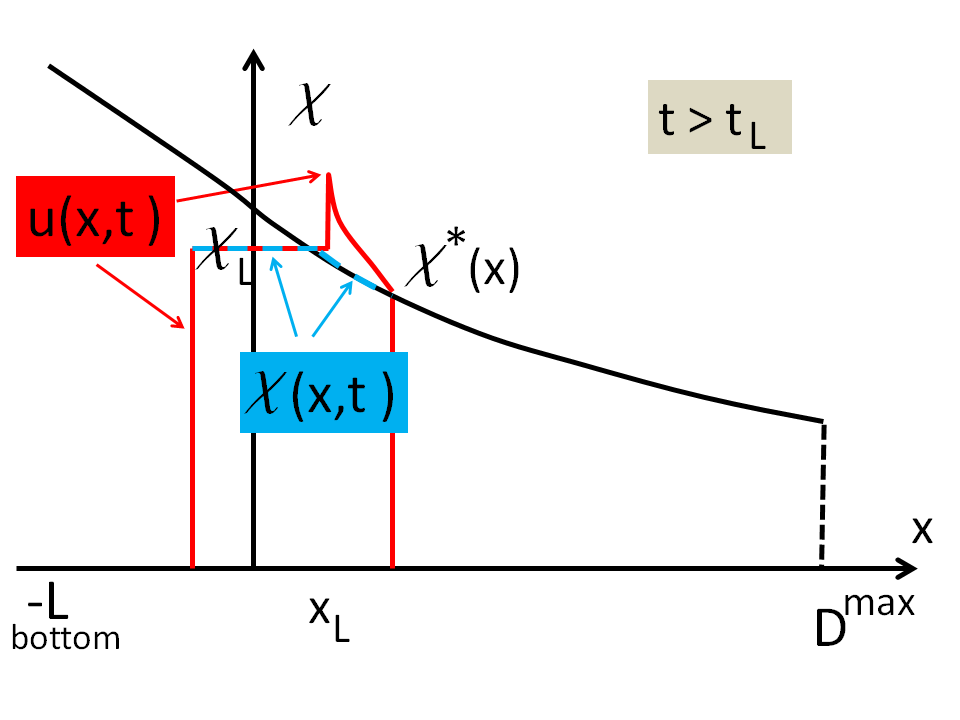}
\centerline{\hfill (c) \hfill (d) \hfill}
\caption{\label{fig:physical} (a) Example of $\chi^*$ for a typical
  reservoir as in \cite{LF08}. (b)-(d) Example in
  Section~\ref{sec:example} at $t=0$, $t=t_L$, and $t>t_L$,
  respectively.}
\end{figure}

We first discuss the applications of maximum and comparison principles
and then consider an analytical solution to a simplified case of
\eqref{bs00} with pure advection in $N=1$. Such a scenario arises when
the system changes rapidly away from hydrostatic equilibrium and when
diffusion is negligible compared to advection.  It also presents the
``worst case scenario'' from the point of view of analysis while it
simultaneously accounts for the largest possible accumulation of
hydrate. The scenario leads to hydrate saturations exceeding $1$ which
is unphysical and can be considered a ``blow-up''. We discuss whether
this blow-up can be anticipated or prevented by the maximum or comparison
estimates.

In all of the examples below we assume 
\ba
\label{eq:chidecreasing}
\chi^{*}(x) \text{ is a smooth non-increasing function in } G,
\ea
which is consistent with typical phase behavior in subsea sediments
\cite{LF08}; see Figure~\ref{fig:physical}(a). 
For simplicity we
consider $N=1$ and that the reservoir
\ba
\label{eq:G}
G \eqdef (0,D^{max})
\ea
has its bottom at $x=0$ and its top $x = D^{max}$ near the seafloor.
We also assume constant porosity \eqref{eq:phi-one} and homogeneous
boundary conditions.

\subsection{Application of maximum and comparison principles}

We consider two main examples of purely diffusive and purely advective
transport. Both are included in the theory.

We have two tools to obtain estimates on a solution of the stationary
problem \eqref{semi} which result eventually in those for the
evolution problem \eqref{eq:bs00}. The first is the comparison
principle \eqref{L1-order} of Theorem~\ref{BS-T1} which bounds one
solution by another solution. The difficulty here is to choose
solutions which provide useful estimates.
The second tool follows the maximum estimate in
Corollary~\ref{BS0-max} which provides bounds which are not
solutions. If $\beta$ is independent of $x$, the constants are
arbitrary, but in the $x$-dependent case the bound is a function
chosen from the {level set} of $\beta^{-1}(x,\cdot)$.

\subsubsection{Purely diffusive case}
Let $D_l^M >0,\; q=0$. 

The well-posedness result of Proposition~\ref{prop:evolution} applies
to the operator $A=-D_l^M \frac{d^2}{dx^2} $ with
domain
\bas
\Dom(A)=\{ v \in W^{2,1}(G): v(0)=v(D^{max})=0\},
\eas
and one can see that it satisfies the hypotheses of the
Theorem~\ref{BS-T1} on $L^1(G)$. Thus there is a unique $C^0$
solution to the problem \eqref{geneq}. 

We would like to use Remark~\ref{rem:bounds} with $v_2 = \chi^*$, but
the assumption \eqref{eq:bounds} requires that $\chi^* \in \Dom(A)$
and $A\chi^* \geq 0$; it allows for sink terms but no
sources. 
Consider the particular
case
\begin{multline}
\label{eq:diff}
\text{affine, decreasing }\chi^*,\text{ initial data }u_0 \leq R,
\\
\text{ with homogeneous boundary conditions, pure diffusion with sinks}.
\end{multline}
Let $\delta > 0$ and choose $v_2 \in \Dom(A)$ to be concave, $0 \le
v_2 \le \chi^*$, and $v_2(x) = \chi^*(x)$ for $x \in
(\delta,D^{max}-\delta)$.  Select $u_2 \in \beta(\cdot,v_2)$ by
$u_2(x) = R$ for $x \in (\delta,D^{max}-\delta)$ and $u_2 = v_2$
otherwise in $G$.  Thus, if $0 \le u_0 \le u_2$ and $F \le 0$, we
obtain from Remark~\ref{rem:bounds} that $0 \le u(t) \le u_2 \le R$
and $\chi(t) \le v_2$ in $G$, so we have a (physical) solution to
\eqref{bs00} that satisfies \eqref{eq:satminmax} for all time.
Note this case was obtained independently (in $H^{-1}(G)$) in \cite{GMPS14}. 
An extension of \eqref{eq:diff} is possible for a concave $\chi^*$,
which is, however, a nonphysical situation.

\subsubsection{Purely advective case} 
Now consider $D_l^M =0, q = const>0$. We develop an explicit analytical
solution for this case below and concentrate on the estimates first.
To apply the well-posedness result of Proposition~\ref{prop:evolution}
we see that the abstract formulation includes the operator
$A=q\frac{d}{dx}$ with domain
\ba
\label{eq:DA}
\Dom(A)=\{ v \in W^{1,1}(0,D^{max}): v(0)=0\},
\ea
which satisfies the hypotheses of Theorem~\ref{BS-T1} on
$L^1(G)$. Thus there is a unique $C^0$ solution to the problem; it is
given in Lemma~\ref{lem:example} below.

Now we would like to apply the maximum estimate to give a bound $b(x) \leq
R$ which would yield \eqref{eq:satminmax}.  Choose $k\geq 0$ and $b(x)
\in \beta(x,k)$. Then $b(x) \leq R$ implies $k \leq \chi^*(x)$ for
every $x \in (0,D^{max}]$. Thus $b(x)=k$ for all $x \in G$. But
the largest possible $k=\chi^*(D^{max}) \leq \chi^*(x)$ by
\eqref{eq:chidecreasing}, so we obtain
\ba
S(x,t)=0, (x,t) \in G \times (0,T).
\ea
Thus the problem with 
\begin{multline}
\label{eq:adv}
\text{non-increasing }\chi^*,\text{ initial data }u_0 \leq  \chi^*(D^{max}),
\\
\text{null boundary condition at 0, only advection, and no sinks/sources},
\end{multline}
is well posed with the original $\beta$. While physical, there is no
hydrate formation in this example with such data.

\medskip
To use the comparison principle to get an upper bound $b(x)=R$ for a
more realistic example with hydrate formation, i.e., $S(x,t)>0$ for
some $x,t$, we would like to choose $v_2 = \chi^*$. However, $\chi^* \not
\in \Dom(A)$ given by \eqref{eq:DA} because it does not satisfy the
boundary condition. However, if we truncate it linearly so
that
\bas
v_2(x) \eqdef \left\{ \begin{array}{ll}\chi^*(x),\, x\geq x_0>0,\\
\chi^*(x_0)\frac{x}{x_0},\; 0<x<x_0,
\end{array}\right.
\eas
then we can choose 
\bas
b(x) \eqdef \left\{ \begin{array}{ll}R,\, x\geq x_0>0,\\
\chi^*(x_0)\frac{x}{x_0},\; 0<x<x_0.
\end{array}\right.
\eas
These would give a good bound except that 
\bas
Av_2(x) =\left\{ \begin{array}{ll}q\partial_x \chi^*(x),\, x\geq x_0>0,
\\
q \chi^*(x_0),\; 0<x<x_0\,,
\end{array}\right.
\eas
and the assumption $Av_2 \geq 0$ holds only if $q\partial_x \chi^*\geq 0$.  This
requires, for the flow towards the ocean floor ($q>0$), to have
$\chi^*$ to be nondecreasing, which is unphysical, a clear contradiction
with \eqref{eq:chidecreasing}.

Alternatively, one can have the profile as shown in
Figure~\ref{fig:physical}, but with flux $q<0$ working towards the
bottom of the reservoir (which requires boundary condition to be
defined at $x=D^{max}$ and not at $x=0$). In summary, this case is
\begin{multline}
\label{eq:adv2}
\text{decreasing }\chi^*,\text{ initial data }u_0 \leq  R,
\\
\text{ with null boundary condition on right, 
 only advection, no sinks, } q <0. 
\end{multline}
This case is well-posed without an extension of $\beta$ but is rarely
seen in practice.
On the other hand, the case 
\begin{multline}
\label{eq:adv3}
\text{decreasing }\chi^*,\text{ initial data }u_0 \leq R,
\\
\text{ null boundary condition on left, 
 only advection, no sinks,} q >0,
\end{multline}
requires an ad-hoc extension $\beta\rightarrow \bar{\beta}$ for
well-posedness and may have unphysical solutions, since
\eqref{eq:satminmax} cannot be guaranteed with the maximum and
comparison principles derived in Section~\ref{analysis}. 

These findings are consistent with the analytical solution we derive
below for the case \eqref{eq:adv3} in which a ``blow-up'' occurs with
$S>1$ at some critical time $t_*$ in violation of
\eqref{eq:satminmax}. However, one can still extend $\beta \rightarrow
\bar{\beta}$ beyond some possible maximum value of $\chi^*,u$ which
depends on the data, and have well-posedness of the problem with
$\bar{\beta}$ producing solutions which are unphysical past $t_*$. We
interpret this as a case for which we have only local in time physical
solution, and simultaneously one in which the model itself becomes
unphysical.

\subsection{Analytical solution for advection case}
\label{sec:example}
We calculate the solution to \eqref{bs00} in the case \eqref{eq:adv3} with
constant flux input. We set $G=(-L,D^{max})$ for some $L>0$ and consider
\begin{subequations} \label{riemann}
\begin{eqnarray}
\label{eq:mh1d}
\partial_t{u}
+ \partial_x (q \chi)=0,\ x \in G,\ t>0
\end{eqnarray}
in which $u = (1-S) \chi + R S \in \beta(x,\chi)$ is determined by \eqref{eq:constraint}.
The initial condition for the problem is
\begin{eqnarray}
\label{eq:1dinit}
u(x,0) = \left\{ \begin{array}{cc}\chi_L,& x\leq 0,\\0,& x>0\end{array}\right.
= \chi_L H(-x),
\end{eqnarray}
\end{subequations}
where $H$ is the Heaviside function. This describes the physical
situation in which no methane is initially present in the reservoir
$(0,D^{max})$, but as time progresses, a uniform pulse of methane of a
fixed concentration and duration enters the reservoir at the left
boundary and gets transported towards the right (upper) boundary into
the reservoir.
\begin{remark}
\label{rem:blob}
The boundary condition at $x=0$ implicit in \eqref{eq:1dinit} is
inhomogeneous and thus not covered by the theory in
Section~\ref{analysis}, but it can be included by posing the problem
on $(-L,D^{max})$ with $u(-L,t) = 0$ as indicated, and setting $u_0(x)
=\chi_LH(-x)$.
\end{remark}

We assume that
$
\chi_L < \chi^*(0),
$
i.e., the incoming methane is all dissolved in the water. There is no
outflow boundary condition on the right end for this first order
equation \eqref{riemann}.

\begin{lemma}
\label{lem:example}
Assume 
\ba
\label{eq:max}
\chi^*(0) > \chi_L \geq \min\{\chi^{*}(x):\ x \in G\}.
\ea
Let $x_L$ be the unique last point where
\ba
\label{eq:xl}
\chi_L=\chi^*(x_L), 
\ea
and define the zones 
\begin{subequations}
\label{eq:zones}
\ba
G_-(t)&\equiv& \{x: 0 < x < \min (qt, x_L)\},\\ 
G_0(t) &\equiv& \{x: x_L < x < qt\},\\
G_+(t)&\equiv& \{x: x>qt\}.
\ea
\end{subequations}
Then the solution $\chi(x,t),S(x,t)$ to \eqref{riemann} is given by
\begin{subequations}
\ba
\label{eq:chi}
\chi(x,t) &=& \min(\chi_L,\chi^{*}(x)) H(qt-x),
\\
\label{eq:sdiscrete}
S(x,t) &=& -\frac{(t-t_x)^+ q \partial_x \chi^{*}(x)}
              {R-\chi^{*}(x)}, \; x \in G_0(t),
\\
S(x,t) &=&0, \; x \in G_-(t) \cup G_+(t),
\ea
\end{subequations}
where $t_x \equiv \frac{x}{q}$ is the breakthrough time for each
position $x$, that is, the first time at which methane is present at
$x$.
\end{lemma}
\begin{proof}
As time $t$ increases, the methane enters the reservoir and much of it
is transported upwards towards the ocean floor located at $x=D^{max}$
and escapes there. However, some of the methane remains trapped in the
reservoir in the form of hydrate because of \eqref{eq:max}.
Since $\chi^*$ is monotone decreasing, there is a unique last point
$x_L \in \bar{G}$ for which \eqref{eq:xl} holds.
The methane which enters $G$ from the left travels up to $x_L$ as a
travelling wave $\chi(x,t)=\chi_L H(qt-x),\;\; 0< x < x_L$,
with speed $q$. 
The dissolved amount $\chi(x,t)$ does not exceed $\chi^*(x),x \leq
x_L$, thus we have $S(x,t)=0, u(x,t)=\chi(x,t),\,0<x<x_L$. 

In summary, for $tq \geq x$ or $t\geq t_x$, we have
$ \chi(x,t)=\min(\chi_L,\chi^{*}(x)) $ and there is no methane
$u(x,t)=\chi(x,t)=0$ ahead of the travelling wave where $tq<x$ ($t<t_x$). This
is concisely written as \eqref{eq:chi}.

One can thus distinguish three zones \eqref{eq:zones}.  
Note that $G_-(t)$ and $G_0(t)$ are empty for $qt < x_L$. While the
right boundaries of $G_-$ and $G_0$ travel with speed $q$, the left
boundary $x_L$ of $G_0$ is a {\em free boundary} determined by the
solution.

Now $u(x,t)$ must be partitioned between the flowing dissolved methane
advected towards the right and the stationary hydrate phase in $G_0$
with saturation $S$. Furthermore, $u(x,t)$ increases due to the
continuous supply of gas advected from the left. We have by
\eqref{eq:defu} and \eqref{eq:chi} that
\ba
\label{eq:u}
u(x,t) = (1-S)\chi^*(x)H(qt-x) + RS, \; x>x_L,
\ea
and we can formally differentiate in time to get
\ba \partial_t u = \partial_t S(x,t) (R - \chi^*(x)H(qt-x))+(1-S)
\chi^*(x)q \delta(qt-x). \ea
Here we have used Dirac $\delta$ for
$\partial_tH$. Differentiating \eqref{eq:chi} in $x$ we have
\bas  \partial_x \chi =  \partial _x \chi^*(x) H(qt-x) -  \chi^*(x)
\delta(qt-x). 
\eas 
Since {$S(x,t)\delta(qt-x)=0$}, substituting these into the
conservation law \eqref{eq:mh1d} yields
\ba
\label{eq:s}
S_t(R-\chi^*(x)H(qt-x))+q\partial_x \chi^*(x)H(qt-x)=0,\ x > x_L.
\ea
In $G_- \cup G_+$ clearly $S\equiv 0$ and $S_t=0$ since these regions
are undersaturated. 

It remains to calculate $S(x,t), x \in G_0(t)$. From
\eqref{eq:s} we have 
\bas
S_t=-\frac{q\partial_x \chi^*(x)}{(R-\chi^*(x))},\; x \in G_0(t),
\eas
and integrating in time from $t_x$ to $t$, with $S(x,t_x)= 0$  gives 
\eqref{eq:sdiscrete}. 
\end{proof}
\begin{remark}
Another way to calculate \eqref{eq:sdiscrete} is to notice that
\ba
\label{eq:omega0}
\partial_t u = - q \partial_x \chi^{*}(x), \; x \in G_0(t),
\ea
which is exactly \eqref{eq:mh1d} written in $G_0(t)$.  Integrating the
right hand side in time we have
\begin{eqnarray}
u(x,t) - u(x,t_x) = -(t-t_x)q \partial_x \chi^{*}(x), \; x \in G_0(t)
\end{eqnarray}
With \eqref{eq:defu} we now see 
\begin{multline}
(1-S(x,t))\chi^*(x)+RS(x,t)
- (1-S(x,t_x))\chi^*(x)+RS(x,t_x)
\\
=
-(t-t_x) q \partial_x \chi^{*}(x), \; x \in G_0(t).
\end{multline}
However, $S(x,t_x)=0$ thus we further simplify to obtain
\ba
S(x,t)(\chi^{*}(x)-R) =
 (t-t_x)q \partial_x \chi^{*}(x), \; x \in G_0(t).
\ea
which is the same as \eqref{eq:s}.
\end{remark}

From these, it follows that if the pulse of methane is not too large, i.e.,
\begin{equation} \label{safe}
\frac{L(-\partial_x \chi^*(x))}{R - \chi^*(x)} \leq S_0 < 1,\ x \ge x_L,
\end{equation}
then $S(x,t) < 1$ and the constraint \eqref{eq:satminmax} is satisfied.

Otherwise, from \eqref{eq:s} we see that $S$ reaches $1$ at $t=t_*$ given by 
\ba
\label{eq:ts}
t_* = t_x + \frac{\chi^{*}(x)-R}{q \partial_x \chi^{*}(x)}.
\ea
This time can be calculated for a given $\chi^{*}$ with known $R$ and
$q$. Note that by \eqref{eq:rdef}, \eqref{eq:defu}, $q>0$,
\eqref{eq:chidecreasing}, we have $t_*>t_x$.

At $t>t_*$, we have that $S>1$, i.e., \eqref{eq:satminmax} is
violated.  We have thus demonstrated an important observation which we
shall discuss in view of the maximum principles discussed earlier.
\begin{corollary}
\label{cor:unphysical}
The solution satisfies \eqref{eq:satminmax} if \eqref{safe}
holds. Otherwise, there is no physically meaningful solution to
\eqref{riemann} for $t>t_*$ where $t_*$ is given by \eqref{eq:ts}.
\end{corollary}

\section{The Coupled Transport-Pressure System}  \label{sys}

Now we discuss the system consisting of the transport (saturation)
equation \eqref{bs00} solved for methane solubility $\chi$, and
methane content $u$. In what follows we assume that all variables can
be defined pointwise at every $x$. (More formally, we use their
regularizations as described below). We also assume that there is
nontrivial diffusion present in the problem, so elliptic regularity
applies.

We recall that from
\eqref{eq:defu} we have 
\ba
\label{eq:satu}
S= \frac{(u-\chi^*)^+}{R-\chi^*},
\ea
that is, $S$ is a Lipschitz function of $u$. (Obviously,
\eqref{eq:satu} is true at every $x \in G,t>0$).

The transport model \eqref{bs00} is coupled with the time-independent
pressure equation \eqref{darcy-e} solved for $p$ and $\q$. The system
\eqref{bs00} and \eqref{darcy-e} is fully coupled because the
permeability $\kappa(x,S(x,t))$ varies with the hydrate
saturation. One could consider the more general pressure equation
\eqref{darcy_c}, but this will not be done here.

The pressure equation \eqref{darcy-e} is elliptic in $p$ and is easily
resolved at any fixed time. Since $\kappa$ varies with $S$, it varies with
time, thus so do $p=p(t)$ and $\q=\q(t)$.  But the linear operator $A$
in the abstract formulation of the transport equation contains $\q$
and thus is time dependent. As such, the coupled case is not covered
directly by the theory developed in Section~\ref{analysis}.

We consider therefore a time-staggered coupling in the system, which
allows the treatment of transport and pressure components separately. 

\medskip
Let $0=t_0<t_1 < \ldots t_n < \ldots t_N=T$ be a sequence of discrete
time intervals. Let $u_0=u(t_0) \in \Dom(A)$. 

For $n=0,1,2,\ldots $ we proceed as follows.

(1) For each $n$, given $u(t_n)$, we can calculate $S(u(t_n))$ using
\eqref{eq:satu}.  Then we can calculate $\kappa(t_n) =
\kappa(S(u(t_n)))$ from \eqref{eq:ksat}. Next we solve the
quasi-static pressure equation \eqref{darcy-e} for $\q(t_n)$.

(2) Given $u(t_n),\, \q(t_n)$ we solve \eqref{bs00} for $u(t) \in
\beta(x,\chi(t)),\ t_n \le t \le t_{n+1}$.

(3) Return to step (1) with $n$ replaced by $n+1$. 

\medskip

This procedure depends on the exchange of data between \eqref{bs00}
and \eqref{darcy-e}. However the solution of the saturation equation
\eqref{bs00} is rather weak, and we only have $u(t) \in L^1(G)$ and
$\chi(t) \in \Dom(A_1)$. We alter (2) by regularizing \eqref{bs00}
with the Lipschitz $\beta_\lambda$ and $A_\eps = A + \eps I$, as in
the proofs of Section~\ref{analysis}, so the corresponding solutions
$u_{\lambda,\eps}$ are good approximations of the solution $u$ of
\eqref{bs00} when $\eps$ and $\lambda$ are small.

Since $S(u_{\lambda,\eps})$ is smooth, the coefficient
$\kappa_{\lambda,\eps} = \kappa(S(u_{\lambda,\eps}))$ is also smooth.
From the pressure equation \eqref{darcy-e} we know that the
corresponding $\q_{\lambda,\eps} \in H^1_{div}(G)$, but it will be
smoother for the approximating equation. Typically, we can appeal to
classical regularity results \cite{KindStam00,GilbargTrud83} to see
that if $\kappa_{\lambda,\eps} \in C^{0,\gamma}(G)$ with $0 < \gamma <
1$, then the solution of the elliptic equation \eqref{darcy-e}
satisfies $p_{\lambda,\eps}(t) \in W^{2,s}(G) \cap C^{1,\gamma}(G)$
for some $s \ge 1$, so the flux $\q_{\lambda,\eps}(t) \in H^{1,s}(G)
\cap C^{0,\gamma}(G)$. (Of course, in the $N=1$ case, the flux is
constant and determined by boundary conditions.)

Conversely, given $q_{\lambda,\eps}(t_n)$ we want to see if it is
smooth enough to construct $A(t_n)$ and solve \eqref{bs00} and close
the loop, so that the exchange of data in the iterative coupling is
meaningful. For $N=1$ it is obviously true. For $N>1$ we only have
$\q_{\lambda,\eps}(t_n) \in C^{0,\gamma}(G), \gamma<1$ which falls
somewhat short of the formally required condition $\q(t_n) \in
C^{1}(G)$. However, this shortcoming can be avoided by additional
minor regularization.

It is understood that $\kappa \ge \kappa_0 > 0$ in $G$. If this is
violated, then $\q \to \0$ and a geomechanics component is required
for the model.

\section{Conclusions}
\label{sec:conclusions}

In this paper we have extended previous work \cite{GMPS14} on a
diffusive model of methane transport with hydrate formation to the
case which includes advective transport coupled to an appropriate pressure
equation. The well-posedness of the model is formulated in
Proposition~\ref{prop:evolution} along with some comparison and
maximum principles that work for various practical cases. However, for
other cases, an extension of data required for well-posedness may
produce unphysical solutions, where hydrate saturation (volume
fraction) exceeds one.

This result was apparent from both the maximum and comparison
principles that we constructed, as well as from an analytical solution
to an initial-boundary-value problem that we constructed for the
purely advective case in $N=1$. Thus there arise the questions of the
relevence of the analysis and the range of application of model.

In fact, it is the oversimplifications of the model that lead to the
unphysical case where \eqref{eq:satminmax} is violated. When the entire
pore space is clogged by hydrate, as when $S=1$, the fluid ceases to
flow, thus $q=0$. Since a one-dimensional model does not allow it, it
is therefore producing unphysical solutions. Even before $S=1$, the
model coupled to the pressure equation must account for the decrease
of $\kappa=\kappa(x,S,p(x))$, thus increased pressure and fracturing.
Eventually, the transport and pressure equation model should be
extended to include a mechanics component with deformation or
fracturing. This is a subject of ongoing work.

Other extensions currently underway include accounting for nonconstant
salinity $\chi_l^S$, and temperature, and for more general
thermodynamics such as that for when water is not available for
hydrate formation.

\bibliographystyle{unsrt}
\bibliography{master,mpesz,showalter}

\def\cprime{$'$} \def\cprime{$'$} \def\cprime{$'$}
\begin{thebibliography}{10}

\bibitem{NRuppel03}
J.~Nimblett and C.~Ruppel.
\newblock Permeability evolution during the formation of gas hydratees in
  marine sediments.
\newblock {\em Journal of Geophysical Research}, 108:B9, 2420, 2003.

\bibitem{Sloan}
E.D. Sloan and C.~A. Koh.
\newblock {\em Clathrate Hydrates of Natural Gases}.
\newblock CRC Press, third edition, 2008.

\bibitem{TT04}
M.E. Torres, K.~Wallmann, A.M. Tr\'ehu, G.~Bohrmann, W.S. Borowski, and
  H.~Tomaru.
\newblock Gas hydrate growth, methane transport, and chloride enrichment at the
  southern summit of {Hydrate Ridge, Cascadia margin off Oregon}.
\newblock {\em Earth and Planetary Science Letters}, 226(1-2):225 -- 241, 2004.

\bibitem{KoreaHydrates}
Japan completes first offshore methane hydrate production test—methane
  successfully produced from deepwater hydrate layers.
\newblock {\em Fire in the Ice, Methane Hydrate Newsletter}, 13(2), 2013.

\bibitem{XuRuppel99}
W.~Xu and C.~Ruppel.
\newblock Predicting the occurence, distribution, and evolution of methane
  hydrate in porous marine sediments.
\newblock {\em Journal of Geophysical Research}, 104:5081--5095, 1999.

\bibitem{LF08}
X.~Liu and P.~B. Flemings.
\newblock Dynamic multiphase flow model of hydrate formation in marine
  sediments.
\newblock {\em Journal of Geophysical Research}, 112:B03101, 2008.

\bibitem{PTT10}
M.~Peszy\'nska, M.~Torres, and A.~Tr\'ehu.
\newblock Adaptive modeling of methane hydrates.
\newblock In {\em International Conference on Computational Science, ICCS 2010,
  Procedia Computer Science, available online {via
  www.elsevier.com/locate/procedia} and {www.sciencdirect.com}}, volume~1,
  pages 709--717, 2010.

\bibitem{PIMA11}
M.~Peszynska.
\newblock Methane in subsurface: mathematical modeling and computational
  challenges.
\newblock In Clint Dawson and Margot Gerritsen, editors, {\em IMA Volumes in
  Mathematics and its Applications 156, Computational Challenges in the
  Geosciences}. Springer, 2013.

\bibitem{DD11}
Hugh Daigle and Brandon Dugan.
\newblock Capillary controls on methane hydrate distribution and fracturing in
  advective systems.
\newblock {\em Geochemistry, Geophysics, Geosystems}, 12(1), 2011.

\bibitem{JungSantamarina}
Jong-Won Jung and J.~Carlos Santamarina.
\newblock Hydrate formation and growth in pores.
\newblock {\em Journal of Crystal Growth}, 345, 2012.

\bibitem{YunSantamarina}
TaeSup Yun and J.Carlos Santamarina.
\newblock Hydrate growth in granular materials: implication to hydrate bearing
  sediments.
\newblock {\em Geosciences Journal}, 15(3):265--273, 2011.

\bibitem{Tohidi2001}
Bahman Tohidi, Ross Anderson, M~Ben Clennell, Rod~W Burgass, and Ali~B
  Biderkab.
\newblock Visual observation of gas-hydrate formation and dissociation in
  synthetic porous media by means of glass micromodels.
\newblock {\em Geology}, 29(9):867--870, 2001.

\bibitem{ChinaHydrates}
Guangxue Zhang, D., Shengxiong Yang, Ming Zhang, Jinqiang Liang, Jingan LU,
  Melanie Holland, Peter Schultheiss, and ScienceTeam GMGS2.
\newblock Gmgs2 expedition investigates rich and complex gas hydrate
  environment in the south china sea.
\newblock {\em Fire in the Ice, Methane Hydrate Newsletter}, 14(1), 2014.

\bibitem{GMPS14}
Nathan~L. Gibson, F.~Patricia Medina, Malgorzata Peszynska, and Ralph~E.
  Showalter.
\newblock Evolution of phase transitions in methane hydrate.
\newblock {\em J. Math. Anal. Appl.}, 409(2):816--833, 2014.

\bibitem{TorresICGH}
Marta~E. Tores, Ji-Hon Kim, Ji-Young Choi, Byong-Jae Ryu, Jang-Jun Bahk,
  Michael Riedel, Timothy Collet, WeiLi Hong, and Miriam Kastner.
\newblock Occurrence of high salinity fluids asociated with masive
  near-seafloor gas hydrate deposits.
\newblock In {\em Procedings of the 7th International Conference on Gas
  Hydrates (ICGH 201), Edinburgh, Scotland, United Kingdom, July 17-21, 2011},
  2011.

\bibitem{Brezis73}
H.~Br{\'e}zis.
\newblock {\em Op\'erateurs maximaux monotones et semi-groupes de contractions
  dans les espaces de {H}ilbert}.
\newblock North-Holland Publishing Co., Amsterdam, 1973.
\newblock North-Holland Mathematics Studies, No. 5. Notas de Matem{\'a}tica
  (50).

\bibitem{EkelandTemam99}
Ivar Ekeland and Roger T{\'e}mam.
\newblock {\em Convex analysis and variational problems}, volume~28 of {\em
  Classics in Applied Mathematics}.
\newblock Society for Industrial and Applied Mathematics (SIAM), Philadelphia,
  PA, english edition, 1999.
\newblock Translated from the French.

\bibitem{Rockafellar68}
R.~T. Rockafellar.
\newblock Integrals which are convex functionals.
\newblock {\em Pacific J. Math.}, 24:525--539, 1968.

\bibitem{Rockafellar69}
R.~T. Rockafellar.
\newblock Measurable dependence of convex sets and functions on parameters.
\newblock {\em J. Math. Anal. Appl.}, 28:4--25, 1969.

\bibitem{Showalter97}
R.~E. Showalter.
\newblock {\em Monotone operators in {B}anach space and nonlinear partial
  differential equations}, volume~49 of {\em Mathematical Surveys and
  Monographs}.
\newblock American Mathematical Society, Providence, RI, 1997.

\bibitem{KindStam00}
David Kinderlehrer and Guido Stampacchia.
\newblock {\em An introduction to variational inequalities and their
  applications}, volume~31 of {\em Classics in Applied Mathematics}.
\newblock Society for Industrial and Applied Mathematics (SIAM), Philadelphia,
  PA, 2000.
\newblock Reprint of the 1980 original.

\bibitem{ItoKun08}
Kazufumi Ito and Karl Kunisch.
\newblock {\em Lagrange multiplier approach to variational problems and
  applications}, volume~15 of {\em Advances in Design and Control}.
\newblock Society for Industrial and Applied Mathematics (SIAM), Philadelphia,
  PA, 2008.

\bibitem{Ulbrich11}
Michael Ulbrich.
\newblock {\em Semismooth {N}ewton methods for variational inequalities and
  constrained optimization problems in function spaces}, volume~11 of {\em
  MOS-SIAM Series on Optimization}.
\newblock Society for Industrial and Applied Mathematics (SIAM), Philadelphia,
  PA; Mathematical Optimization Society, Philadelphia, PA, 2011.

\bibitem{CranEvans75}
Michael~G. Crandall and L.~C. Evans.
\newblock On the relation of the operator {$\partial /\partial s+\partial
  /\partial \tau $} to evolution governed by accretive operators.
\newblock {\em Israel J. Math.}, 21(4):261--278, 1975.

\bibitem{Evans77}
Lawrence~C. Evans.
\newblock Application of nonlinear semigroup theory to certain partial
  differential equations.
\newblock In {\em Nonlinear evolution equations ({P}roc. {S}ympos., {U}niv.
  {W}isconsin, {M}adison, {W}is., 1977)}, volume~40 of {\em Publ. Math. Res.
  Center Univ. Wisconsin}, pages 163--188. Academic Press, New York, 1978.

\bibitem{Benilan72b}
Philippe B{\'e}nilan.
\newblock Solutions int\'egrales d'\'equations d'\'evolution dans un espace de
  {B}anach.
\newblock {\em C. R. Acad. Sci. Paris S\'er. A-B}, 274:A47--A50, 1972.

\bibitem{BenCranSacks88}
Philippe B{\'e}nilan, Michael~G. Crandall, and Paul Sacks.
\newblock Some {$L^1$} existence and dependence results for semilinear elliptic
  equations under nonlinear boundary conditions.
\newblock {\em Appl. Math. Optim.}, 17(3):203--224, 1988.

\bibitem{CranLigg71}
M.G. Crandall and T.M. Liggett.
\newblock Generation of semi-groups of nonlinear transformations on general
  {Banach} spaces.
\newblock {\em Amer. J. Math.}, 93:265--293, 1971.

\bibitem{PeszShow98}
M.~Peszy{\'n}ska and R.~E. Showalter.
\newblock A transport model with adsorption hysteresis.
\newblock {\em Differential Integral Equations}, 11(2):327--340, 1998.

\bibitem{ShowLitHorn96}
Ralph~E. Showalter, Thomas~D. Little, and Ulrich Hornung.
\newblock Parabolic {PDE} with hysteresis.
\newblock {\em Control Cybernet.}, 25(3):631--643, 1996.
\newblock Distributed parameter systems: modelling and control (Warsaw, 1995).

\bibitem{LittShow95}
T.~D. Little and R.~E. Showalter.
\newblock The super-{S}tefan problem.
\newblock {\em Internat. J. Engrg. Sci.}, 33(1):67--75, 1995.

\bibitem{show84b}
R.~E. Showalter.
\newblock A singular quasilinear diffusion equation in {$L\sp{1}$}.
\newblock {\em J. Math. Soc. Japan}, 36(2):177--189, 1984.

\bibitem{lake}
L.~W. Lake.
\newblock {\em Enhanced oil recovery}.
\newblock Prentice Hall, 1989.

\bibitem{BenGharbia}
Ibtihel~Ben Gharbia and Jerome Jaffre.
\newblock Gas phase appearance and disappearance as a problem with
  complementarity constraints.
\newblock {\em Mathematics and Computers in Simulation}, to appear, 2013.

\bibitem{vazquez07}
Juan~Luis V{\'a}zquez.
\newblock {\em The porous medium equation}.
\newblock Oxford Mathematical Monographs. The Clarendon Press Oxford University
  Press, Oxford, 2007.
\newblock Mathematical theory.

\bibitem{DibeShow81}
Emmanuele DiBenedetto and R.~E. Showalter.
\newblock Implicit degenerate evolution equations and applications.
\newblock {\em SIAM J. Math. Anal.}, 12(5):731--751, 1981.

\bibitem{BrezisStrauss73}
Ha{\"{\i}}m Br{\'e}zis and Walter~A. Strauss.
\newblock Semi-linear second-order elliptic equations in {$L^{1}$}.
\newblock {\em J. Math. Soc. Japan}, 25:565--590, 1973.

\bibitem{CranPazy1969}
Michael~G. Crandall and Amnon Pazy.
\newblock Semi-groups of nonlinear contractions and dissipative sets.
\newblock {\em J. Functional Analysis}, 3:376--418, 1969.

\bibitem{GilbargTrud83}
D.~Gilbarg and N.S. Trudinger.
\newblock {\em Elliptic Partial Differential Equations}.
\newblock Springer-Verlag, Berlin-New York, 2nd edition, 1983.

\end{thebibliography}
\end{document}